\documentclass[reqno, 12pt]{amsart} 

\usepackage{amsmath, amssymb} 
\usepackage{amscd}            
\usepackage{amsxtra}          
\usepackage{upref}            

\usepackage[mathscr]{eucal}

\theoremstyle{plain}

\newtheorem{theorem}{Theorem}
\newtheorem{proposition}[theorem]{Proposition}
\newtheorem{lemma}[theorem]{Lemma}
\newtheorem{corollary}[theorem]{Corollary}

\theoremstyle{remark}
\newtheorem{remark}[theorem]{Remark}
\newtheorem{example}[theorem]{Example}

\numberwithin{equation}{section}
\numberwithin{theorem}{section}
\newcommand{\be}%
  {\protect\setcounter{equation}{\value{subsubsection}}}  
\newcommand{\ee}%
  {\protect\setcounter{subsubsection}{\value{equation}}}

\DeclareMathAlphabet\BOONDOX{U}{rsfso}{m}{n}
\newcommand{\N}{{\mathbb N}}
\newcommand{\Z}{{\mathbb Z}}
\newcommand{\Q}{{\mathbb Q}}
\newcommand{\R}{{\mathbb R}}
\newcommand{\C}{{\mathbb C}}

\newcommand{\qp}{{\mathbb Q}_p}
\newcommand{\aq}{{\mathbb A}_{\mathbb Q}}
\newcommand{\ak}{{\mathbb A}_K}

\newcommand{\sltwor}{{\rm SL}_2(\R)}

\newcommand{\glnqv}{{\rm GL}_n({\mathbb Q}_v)}
\newcommand{\glniqv}{{\rm GL}_{n_i}({\mathbb Q}_v)}

\newcommand{\gltwog}{{\rm GL}_2}
\newcommand{\gltwor}{{\rm GL}_2(\R)}
\newcommand{\glthreeg}{{\rm GL}_3}
\newcommand{\glthreer}{{\glthreeg(\R)}}
\newcommand{\gltwoqp}{{\rm GL}_2(\qp)}
\newcommand{\glone}{{\rm GL}_1({\mathbb A}_{\mathbb Q})}
\newcommand{\gltwo}{{\rm GL}_2({\mathbb A}_{\mathbb Q})}
\newcommand{\glthree}{{\rm GL}_3({\mathbb A}_{\mathbb Q})}

\newcommand{\gln}{{\rm GL}_n({\mathbb A}_{\mathbb Q})}

\newcommand{\glntwo}{{\rm GL}_{n-2}({\mathbb A}_{\mathbb Q})}
\newcommand{\glnak}{{\rm GL}_n(\ak)}

\newcommand{\lieg}{{\mathfrak g}}
\newcommand{\ug}{{\mathcal U}}
\newcommand{\liegltwo}{{\mathfrak{gl}}_2(\R)}

\newcommand{\aselp}{{\mathfrak A}}
\newcommand{\asel}{{\aselp}^{\#}}

\newcommand{\aseld}{\asel_d}

\newcommand{\gselp}{\mathfrak{G}}
\newcommand{\gsel}{{\gselp}^{\#}}

\newcommand{\temprep}{\BOONDOX{T}}
\newcommand{\temprepn}{\BOONDOX{T}_n}

\newcommand{\sel}{\mathcal{S}}
\newcommand{\esel}{\BOONDOX{S}^{\#}}
\newcommand{\seld}{\sel_d}
\newcommand{\eseld}{\esel_d}

\newcommand{\ft}[2]{\BOONDOX{F}_{#1}#2}

\newcommand{\aut}{{\mathcal A}^{\circ}}

\newcommand{\autone}{{\mathcal A}_1^{\circ}}

\newcommand{\auttwo}{{\mathcal A}_2^{\circ}}
\newcommand{\auttwoall}{{\mathcal A}_2}
\newcommand{\autthree}{{\mathcal A}_3^{\circ}}
\newcommand{\autthreeall}{{\mathcal A}_3}

\newcommand{\autn}{{\mathcal A}_n^{\circ}}
\newcommand{\autnall}{{\mathcal A}_n}
\newcommand{\autnoneall}{{\mathcal A}_{n-1}}
\newcommand{\autnone}{\autnoneall^{\circ}}
\newcommand{\autntwoall}{{\mathcal A}_{n-2}}
\newcommand{\autnthreeall}{{\mathcal A}_{n-3}}

\newcommand{\autni}{\autniall^{\circ}}
\newcommand{\autniall}{{\mathcal A}_{n_i}}

\newcommand{\whit}{{\mathcal W}}

\newcommand{\reals}{\rm{Re}(s)}

\DeclareMathOperator{\sgn}{sgn}
\DeclareMathOperator{\re}{{\rm{Re}}}

\DeclareMathOperator{\norm}{\|\,\|}
\DeclareMathOperator{\res}{{\rm{Res}}}
\DeclareMathOperator{\sym}{Sym}

\usepackage{scalerel,stackengine}
\stackMath
\newcommand\reallywidecheck[1]{%
\savestack{\tmpbox}{\stretchto{%
  \scaleto{%
    \scalerel*[\widthof{\ensuremath{#1}}]{\kern-.6pt\bigwedge\kern-.6pt}%
    {\rule[-\textheight/2]{1ex}{\textheight}}
  }{\textheight}%
}{0.5ex}}%
\stackon[1pt]{#1}{\scalebox{-1}{\tmpbox}}%
}

\begin{document}


\title{Quotients of $L$-functions: degrees $n$ and $n-2$}
\author{Ravi ~Raghunathan}
\address{Department of Mathematics \\ 
         Indian Institute of Technology Bombay\\
         Mumbai,\enspace  400076\\ India}
\email{raviraghunathan@iitb.ac.in}

\begin{abstract} 
If $L(s,\pi)$ and $L(s,\rho)$ are the Dirichlet series attached to cuspidal automorphic representations $\pi$ and $\rho$ of ${\rm GL}_n({\mathbb A}_{\mathbb Q})$ and ${\rm GL}_{n-2}({\mathbb A}_{\mathbb Q})$ respectively, we show that $F_2(s)=L(s,\pi)/L(s,\rho)$ has infinitely many poles.
 We also establish analogous results for Artin $L$-functions and other $L$-functions not yet proven to be automorphic. Using the classification theorems of \cite{Ragh20} and \cite{BaRa20}, we show that cuspidal $L$-functions of ${\rm GL}_3({\mathbb A}_{\mathbb Q})$ are primitive in ${\mathfrak G}$, a monoid that contains both the Selberg class ${\mathcal{S}}$ and $L(s,\sigma)$ for all unitary cuspidal automorphic representations $\sigma$ of ${\rm GL}_n({\mathbb A}_{\mathbb Q})$.

\end{abstract}
\vskip 0.2cm

\maketitle


\markboth{RAVI RAGHUNATHAN}{Quotients of $L$-functions}


\section{Introduction}\label{introduction}

Let $\aq$ denote the ad\`eles over $\Q$. We denote by
$\autn$ the set of isomorphism classes of unitary cuspidal automorphic representations of $\gln$ and set $\aut=\bigsqcup_n\autn$. Let $\temprepn$ be the isomorphism classes of isobaric sums of elements of $\autni$ such that $\sum_i n_i=n$.
For $\sigma\in \temprep=\bigsqcup_n\temprepn$, let $L(s,\sigma)$ be the (incomplete) $L$-function attached to $\sigma$, and $L(s,\sigma_{\infty})$ its archimedean $L$-factor.  

\begin{theorem} \label{zerothm} Suppose $\pi\in \autn$ and $\rho\in\temprep_{n-2}$. The function $F_2(s)=L(s,\pi)/L(s,\rho)$ has infinitely many poles. Further, if $G_2(s)=L(s,\pi_{\infty})/L(s,\rho_{\infty})$ has a finite number of zeros, $F_2(s)$ has infinitely many poles in $0<\re(s)<1$.
\end{theorem}
Or, infinitely many zeros of $L(s,\rho)$ are not zeros of $L(s,\pi)$ when counted with multiplicity.
The automorphy of the individual $L$-functions in the quotient $F_2$ is not crucial. Theorem \ref{mostgeneral}, and its Corollaries \ref{artin} and \ref{tensprodzeros}, allow us to treat quotients of Artin and tensor product $L$-functions. Theorem \ref{symsquarebyzeta} gives an example of a situation where we can show that $F_2$ has infinitely many poles in $0<\re(s)<1$, even though $G_2$ has infinitely many zeros.

In \cite{Ragh20} 
we defined the set of Dirichlet series $\gselp$ axiomatically (see Section \ref{secdef} for the precise 
definition) which forms a monoid under multiplication,
and contains both the Selberg class $\sel$ and all the standard $L$-functions $L(s,\sigma)$ for 
$\sigma\in \aut$. By hypothesis, the elements of $\gselp$ have at most a finite number of poles in $\C$. We thus obtain the following corollary.
\begin{corollary} \label{generaln} With the notation of Theorem \ref{zerothm}, $F_2\notin \gselp$. 
\end{corollary}
The classification of elements of $\gselp$ of degrees less than $2$ in \cite{Ragh20,BaRa20}, together with the corollary above, yields
\begin{theorem}\label{primthm} If $\pi$ is in $\autthree$, $L(s,\pi)$ is primitive in $\gselp$.
\end{theorem}
Recall that an element $F\ne 1$ in $\gselp$ is primitive if 
$F=F_0F_1$, for $F_0,F_1\in \gselp$,  implies $F_0=1$ or $F_1=1$ ($1$ is the only unit in $\gselp$). Theorem \ref{primthm} is the first instance of establishing the primitivity of $L$-functions of degree greater than $2$. 

In some important cases it is known that $L(s,\pi)$ with $\pi\in \autthree$ belongs to the Selberg class $\sel\subseteq \gselp$. In these cases Theorem \ref{primthm} yields the following corollaries, since primitivity in $\gselp$ implies primitivity in $\sel$.
In the corollaries that follow, it is known that $L(s,\pi)\in \sel$, so the primitivity results in $\sel$ follow from Theorem \ref{primthm}
\begin{corollary}\label{symsquare} Suppose that $\pi$ is is the symmetric square lift
of a representation $\sigma$ in $\auttwo$, where $\sigma$ arises in one of the following ways.
\begin{enumerate}
\item The representation $\sigma$ is associated to a holomorphic cuspidal eigenform of any
congruence subgroup of ${\rm{SL}}_2(\Z)$ and is not dihedral.
\item The representation $\sigma$ is associated to a Maass cuspidal eigenform 
of the full modular group ${\rm SL}_2(\Z)$.
\end{enumerate}
Then $L(s,\pi)$ is primitive in $\sel$.
\end{corollary}
Apart from those covered in the corollary above, there is one other important case.
\begin{corollary}\label{cubiclift} Let $K$ be a cubic extension of $\Q$ and let 
$\chi$ be an id\`ele class character of $\ak$ whose cubic lift lies in 
$\autthree$. Then $L(s,\chi)$ is primitive in $\sel$.
\end{corollary}
The conclusion of Corollary \ref{symsquare} 
is valid for the $L$-functions of Mass cuspidal eigenforms forms of congruence subgroups of small level for which the Selberg Eigenvalue conjecture has been proved (see \cite{BoSt2007} and \cite{BLS2020}, for instance). It is also valid for the $L$-functions of a large class of regular algebraic cuspidal Galois representations for which the $L$-functions are, once again, known to lie in $\sel$.
We also note that we can use the classification results of Kaczorowski and Perelli in \cite{KaPe99} and \cite{KaPe03} for the class $\sel$ (for degrees less that $2/3$) together with Corollary \ref{generaln} to prove the two corollaries above. Thus, they are not dependent on the more general results of \cite{Ragh20} and \cite{BaRa20}.

Prior to the theorems in this paper, primitivity was only know for the $L$-functions $L(s,\pi)$ with $\pi$ in $\autnone$ or $\auttwo$ (see Theorem 6.3 of \cite{Ragh20}).
In these cases the primitivity results are almost direct consequences 
of the theorems in \cite{Ragh20} classifying elements of small degree 
in $\asel$, a class of Dirichlet series containing $\gselp$ (the corresponding theorems for the Selberg class $\sel$ were proved in 
\cite{KaPe99}). In 1996, Ram Murty (and more recently, A. Ivic in 2010) asked me if I could prove that the symmetric square $L$-function of the Ramanujan cusp form was primitive in $\sel$. Corollary \ref{symsquare} of this paper answers the question affirmatively, in the even larger class 
$\gselp$ and, in fact, for all $\pi\in \autthree$.

Theorems \ref{zerothm} and  \ref{mostgeneral} generalise the previous results in three directions. First, there is no restriction on the  conductors of the representations $\pi$ and $\rho$. Second, the theorems apply to the case when $G_2(s)=L(s,\pi_{\infty})/L(s,\rho_{\infty})=R(s)L(s,\tau_{\infty})$, where $\tau_{\infty}$ is a unitary irreducible representation of $\gltwor$ and $R(s)$ is any rational function. In fact, we prove that the assumption that $G_2$ has finitely many zeros implies that it must have this latter form. 
Finally, there is no restricition on the degree $n$.
The theorems of \cite{Ragh99} and \cite{Ragh10} required $\pi$ and $\rho$ to have the same conductors and required $R(s)=1$.
The two theorems also vastly generalise the Theorem of Neurrer and Oliver in \cite{NeOl2020}, where $n=3$, $\pi$ is a symmetric square lift from $GL_2$ and $\rho$ is the trivial character, and a similar result of Hochfilzer and Oliver in \cite{HoOl2022} for Artin $L$-functions. For $\pi\in \autn$ and $\rho\in \autnone$, Booker proved an analogous  result in \cite{Booker2015}, and this was improved upon in \cite{Ragh20}.

In Theorem \ref{mostgeneral}, which is valid for a large number of quotients of $L$-functions which are not known individually to be automorphic,  we will actually prove  that if $F_2$ has a finite number of poles in $\C$ it must be quasi-automorphic, that is, there exist a (unitary)  automorphic representation $\tau\in \auttwoall$ such that we have an equality of partial $L$-functions $F_{2,S}(s)=L_S(s,\tau)$, for a large enough set of finite places $S$. In many cases, including those involving Artin $L$-functions or the tensor product $L$-functions, this will contradict the irreducibility or the pairwise distinctness of these $L$-functions.
For the standard automorphic $L$-functions, this distinctness is a consequence of the theory of Rankin-Selberg convolutions developed by Jacquet, Piatetski-Shapiro and Shalika and, more specifically, the results of Jacquet and Shalika in \cite{JaSh811}. 
For other $L$-functions, we are sometimes able to reduce the situation to the automorphic case, or impose additional conditions to get the desired outcomes.

Recall that one way of proving the automorphy of Dirichlet series is to use
a {\it converse theorem}, that is, a theorem that asserts that a Dirichlet series is automorphic if it satisfies certain analytic properties and satsfies suitable functional equations. Our strategy for proving the automorphy of $F_2$ is to show that if $F_2\in \gselp$ , it necessarily satisfies the hypotheses of the celebrated {\it converse theorem} of Weil (\cite{Weil67}), or its analogues, which we will refer to collectively as $\gltwog$-type converse theorems.

If the function $G_2$ has infinitely many zeros, Theorem \ref{zerothm} follows almost immediately, so we can assume that $G_2$ has only finitely many zeros. {\it A priori}, $G_2$ does not appear to have the form of the factors that arise in $\gltwog$-type converse theorems, where there are no gamma functions in the denominator. An  elementary but careful analysis shows that the factors in the denominator actually cancel with those in the numerator leaving the correct form upto a factor of a rational function. Removal of this factor requires a choice of a suitable test vector for the archimedean integrals associated to the representations of $\gltwor$,
and an invocation of the local functional equation for these integrals. This is one of the main ways in which our proof improves on earlier work. 

A second technical difficulty is overcome using a modification of Booker's theorem in \cite{Booker2003} so that it holds in the class $\gselp$ (and, in fact, in the even larger class $\gsel$): crucially we allow quotients of the usual gamma factors to appear in the functional equation. This will allow us to treat character twists of Dirichlet series of degree 2 satsifying even more general functional equations than the classical Hecke and Maass functional equations (and our treatment gives a uniform proof across different cases; we do not follow the more recent \cite{BFL2022}). 

We should note that a stronger version of Theorem \ref{bookermodthm} of this paper will show that $F_2$ has infinitely many poles in the critical strip without assuming the condition that $G_2$ has finitely many zeros. Our primary aim in this paper was to prove that the $L$-functions of cuspidal automorphic representations of $\glthree$ are primitive for which we require only the weaker Corollary \ref{generaln}. The strengthening of Theorem \ref{bookermodthm} is work in progress.

In addition to the above, we will need to use the stability results for gamma factors proved by Jacquet and Shalika in \cite{JaSh85} in order control the $\varepsilon$-factors that occur in the functional equations of the character twists of $F_2(s)$, as well as a number of other inputs from the theory of automorphic $L$-functions and their integral representations. We will discuss these issues in further detail in Section \ref{org} where we outline the strategy of the proof.

Finally, we note that our techniques are unlikely to generalise further without significant breakthroughs in at least two directions. For instance, for a pair $(\pi,\rho)$ where $\pi\in \autn$ and $\rho\in \autnthreeall$, we would require the analogue of Booker's theorem in \cite{Booker2003} for $L$-functions of degree $3$, which so far seems to be outside the scope of our current technology. If the difference in degrees between $L(s,\pi)$ and $L(s,\sigma)$ is greater than $3$, there is simply no analogue of Weil's converse theorem available which guarantees the automorphy of the $L$-function from the knowledge of only the character twists. Even when the difference in the degrees is exactly three, we will need an improvement on the converse theorem of Jacquet, Piatetski-Shapiro and Shalika  formulated for the $L$-functions of admissible representations. These have Euler factors at each prime $p$ which are reciprocals of polynomials in $p^{-s}$, but our methods require a theorem for series of degree $3$ that may not have this property.

\section{The strategy of the proof and organisation of the paper}\label{org}

We give an outline of our proof of Theorem \ref{zerothm} and also describe the structure of the paper.

In Section \ref{secdef}, we give the precise definitions
of the various classes of Dirichlet series $\sel$, $\esel$, $\gselp$, $\gsel$, $\aselp$, and $\asel$ defined axiomatically in the spirit of the Selberg class $\sel$, and other variants that appear in this paper. 

Assume now that $F_2(s)=L(s,\pi)/L(s,\rho)=\sum_{n=1}^{\infty}a_nn^{-s}\in \gselp$. 
Given the definition of $F_2$, this amounts to assuming that $F_2$ has at most a finite number of poles in $\C$. All the other conditions for membership in $\gselp$ are easily seen to be satisfied because of the properties of $L(s,\pi)$ and $L(s,\rho)$.
If $G_2(s)=L(s,\pi_{\infty})/L(s,\rho_{\infty})$ has infinitely many zeros, it is easy to see that $F_2$ has (infinitely many) poles at those zeros, so we focus on the case when $G_2$ has only finitely many zeros. As mentioned in the introduction, the main task is to establish that the function $F_2$ satisfies the hypotheses of the converse theorem of Weil, or that of its analogues for $L$-functions having archimedean factors arising from principal series representations of $\gltwor$ (recall that the series in Weil's theorem are assumed to have archimedean factors that arise from (limits of) discrete series representations of $\gltwor$). We will refer to the archimedean factors in the two cases together as ``$\gltwog$-type" factors and to the corresponding converse theorems as ``$\gltwog$-type" (converse) theorems. We emphasise that these converse theorems do not require that the Dirichlet series in question have Euler products.

We first show that the twisted series $F_2(s,\chi)=\sum_{n=1}^{\infty}\chi(n)a_nn^{-s}$ which are {\it a priori} only meromorphic are actually entire for all primitive Dirichlet characters $\chi$, a necessary input for the converse theorem. This is carried out in Section \ref{secbookanalog} using the ideas of Booker \cite{Booker2003} but in the more general context of the class $\gsel$. This has the virtue of giving both a more general result as well as avoiding case by case analyses. 

In Section \ref{secautomorph} we describe the local and global $L$-functions and 
associated to automorphic representations of $\gln$. We also record many of the foundational theorems for these $L$-functions primarily due to Jacquet, Piatetski-Shapiro and Shalika via the method of integral representations, Shahidi, Kim-Shahidi and Gelbart-Shahidi using the Langlands-Shahidi method, and the work of Moeglin-Waldspurger which combines both approaches. These allow us to conclude that $L(s,\pi)\in \gselp$ if $\pi\in \temprepn$ and, in fact, satisfy the more restrictive form of the functional equation \eqref{autlfneqn}. We obtain slightly weaker conclusions for the tensor product, exterior square and symmetric square $L$-functions, which will nonetheless suffice for our purposes.

As mentioned above, the $\gltwog$-type converse theorem demands that the archimedean factors $G_2(s,\chi)$ that appear in the functional equation of $F_2(s,\chi)$ should be the archimedean $L$-functions associated to an irreducible admissible representation of $\gltwor$, but {\it a priori}, we can only conclude that the $G(s,\chi)$ are quotients of the gamma functions defined in \eqref{archlfn} and \eqref{twistedarchlfn}. 
The passage from such quotients to the $\gltwog$-type gamma factors is done in two stages which we describe below (recall that we are assuming that $G_2$ has finitely many zeros!).

In Lemma \ref{correctgamma} of Section \ref{redtogltwo} we show that this is true upto a factor of a rational function, that is, if $F_2\in \gselp$, $G(s,\chi_{\infty})=R(s,\chi_{\infty})L(s,\tau_{\infty}\times\chi_{\infty})$ for some irreducible admissible representation $\tau_{\infty}$ of $\gltwor$ and some rational function $R(s,\chi_{\infty})$. 
The arguments involve elementary but careful analyses of the poles of the relevant gamma functions. Then, in Theorem \ref{testvecforps} of Section \ref{archgltwo}, we analyse the archimedean local integrals associated to irreducible admissible representations $\tau_{\infty}$ of $\gltwor$, and show that by choosing suitable test vectors for these integrals we can obtain the factor $P(s,\chi_{\infty})L(s,\sigma_{\infty}\times\chi_{\infty})$ for any polynomial $P(s,\chi_{\infty})$. The (archimedean) local functional equation then allows us to eliminate the factor $R(s,\chi_{\infty})$ from the global functional equation.

Finally, the converse theorem of $\gltwog$-type requires a degree of control on the $\varepsilon$-factors that appear in the functional equation of $F_2(s,\chi)$. This
technical problem can be overcome by twisting $F_2(s)$ by a character $\chi_0$ that is sufficiently highly ramified at a set of finite places (depending on $\pi$ and $\rho$), and then using a stability result due to Jacquet and Shalika for $\varepsilon$-factors. This process is described in Section \ref{secepsilon} and guarantees that the  $\varepsilon$-factors are those that appear in the hypotheses of the converse theorem.

In Section \ref{conversesec} we state the precise $\gltwog$-type converse theorem that we will use. We also state Theorem \ref{mostgeneral} and complete its proof, an easy task, since all the ingredients are already in place. Theorem \ref{zerothm} follows almost immediately. 

 Section \ref{secprim} is devoted to the proof of Theorem\ref{primthm} and its corollaries after reviewing our results from\cite{Ragh20} and \cite{BaRa20} classifying the elements of $\gsel$ of degrees less than $2$. This concludes the main part of the paper.
 
 The remaining sections deal with the applications and extensions of Theorem \ref{zerothm}.
 Section \ref{zerosets} discusses pairs of $L$-functions of symmetric powers as concrete instances of Theorem \ref{zerothm}
 when $G_2$ has finitely many zeros, and how these results improve
 on existing ones. Section \ref{relaxaut} treats Artin $L$-functions
 and other examples of $L$-functions which are not known to be
 automorphic. Finally, in Section \ref{symsquarebyzeta}, we are 
 able to show that $F_2(s)=L(s,\sym^2\tau)/\zeta(s)$ has infinitely
 many poles in the critical strip for $\tau\in \auttwo$, with 
 $\tau_{\infty}$ a discrete series representation of $\gltwor$ of
 even weight. This case is interesting because $G_2$ has infinitely
 many zeros, so it does not satisfy the hypothesis of the second part of Theorem \ref{zerothm}. Nonetheless, the conclusion holds.

\section{Various axiomatically defined classes of Dirichlet series}
\label{secdef}
We recall the definitions of the various classes of Dirichlet series
that appear in this paper. Let $s=\sigma+it\in \C$ and let $F(s)$ be a non-zero meromorphic
function on $\C$. 
We consider the following conditions on $F(s)$.
\begin{enumerate}
\item [(P1)] The function $F(s)$ is given by a
Dirichlet series $\sum_{n=1}^{\infty}\frac{a_n}{n^{s}}$ 
with abscissa of absolute convergence $\sigma_a\ge 1/2$.
\item[(P2)] There is a polynomial $P(s)$ such that $P(s)F(s)$ 
extends to an entire function, and such that in  any
vertical strip $\sigma_1\le \sigma\le \sigma_2$, 
\[
P(s)F(s)\ll e^{|t|^{\rho}}
\]
for some $\rho>0$.
\item[(P3)] There exist a real number $Q>0$, a complex number $\omega$, 
and a function $G(s)$ of the form
\begin{equation}\label{gammafactor}
G(s)=\prod_{j=1}^{r}\Gamma(\lambda_j s+\mu_j)\prod_{j'=1}^{r'}\Gamma(\lambda_{j^{\prime}}^{\prime}s+\mu_{j^{\prime}}^{\prime})^{-1},
\end{equation}
where $\lambda_j, \lambda_{j^{\prime}}^{\prime}\in \R_{>0}$, $\mu_j, \mu_{j^{\prime}}^{\prime}\in \C$, 
and $\Gamma(s)$ denotes the usual gamma function, such that
\begin{equation}\label{fnaleqn}
\Phi(s):=Q^{s}G(s)F(s)=\omega\overline{\Phi(1-\bar{s})}.
\end{equation}
\item[(P4)] The function $F(s)$ can be expressed as a product
$F(s)=\prod_p F_p(s)$, where 
\begin{equation}\label{prod}
\log F_p(s)=\sum_{k=1}^{\infty} \frac{b_{p^k}}{p^{ks}}
\end{equation}
with $\vert b_{p^k}\vert \le Cp^{k\theta}$ for some $\theta>0$ and some
constant $C>0$.
\end{enumerate}
In \cite{Ragh20} we introduced the monoid (under multiplication of functions) of non-zero meromorphic functions $F$ 
satisfying (P1)-(P3) and denoted it by $\asel$. The monoid  of functions $F\in \asel$ also satisfying (P4) is called $\aselp$.

We have also previously introduced the submonoid $\gselp\subset \aselp$ (see \cite{Ragh20}) consisting of Dirichlet series satisfying the following additional properties.
\begin{enumerate}
\item[(G1)] The abscissa of absolute convergence satisfies $\sigma_a\le 1$.
\item[(G3)] The archimedean parameters satisfy the bounds
\begin{equation}\label{gselbound}
\re\left(-\frac{\mu_j}{\lambda_j}\right), 
\re\left(-\frac{\mu_{j^{\prime}}^{\prime}}{\lambda_{j^{\prime}}^{\prime}}\right)<\frac{1}{2}, \,\,1\le j\le r,\,\,1\le j^{\prime}\le r^{\prime}.
\end{equation}
\item[(G4)] The non-archimedean parameters satisfy $\theta<1/2$.
\end{enumerate}
Thus, $F\in \gselp$ if $F$ satisfies (G1), (P2), (G3) and (G4).
Given $F(s)$ satisfying an equation of the form \eqref{fnaleqn}, we define
the degree of $F$ to be
$d_F=2(\sum_{j=1}^{r}\lambda_j-\sum_{j^{\prime}=1}^{r^{\prime}}\lambda_{j^{\prime}}^{\prime})$.
We recall Theorem 4.1 of \cite{Ragh20}.
\begin{theorem}\label{uniquegamma} 
If $F\in \asel$ satisfies 
\eqref{fnaleqn} for two functions
$G^{(1)}(s)$ and $G^{(2)}(s)$ of the form given by \eqref{gammafactor},
let $d^{(1)}=2\sum_{i=1}^{{r}_i}\lambda_i^{(1)}$ and
 $d^{(2)}=2\sum_{i=1}^{{r}_i}\lambda_i^{(2)}$. Then
 $d^{(1)}=d^{(2)}$.  If we further assume that $F\in \gsel$,  $G_1(s)=cG_2(s)$ for some constant $c$. 
\end{theorem}
The theorem says that $d_F$ does not depend on the choice of $G(s)$ in equation \eqref{fnaleqn}. Thus, the notation $\aseld$ for set of series $F$ in $\asel$ with $d_F=d$ is justified.

We recall that the extended Selberg class $\esel$ consists of the series in $\asel$ satisfying
$\sigma_{a}\le 1$, $P(s)=(s-1)^m$ for some integer $m\ge 0$, $r^{\prime}=0$ and $\re(-\mu_j)\le 0$ for $1\le j\le r$. In particular, elements of $\esel$ satisfy (G1) and (G3). The subset of series in $\esel$ with an Euler product satisfying (P4) and (G4) is the Selberg class $\sel$. Thus, we see that $\sel\subset \gselp$. We denote by $\gsel_d$ (resp. $\gselp_d$, $\eseld$, $\seld$) the set of series in $\gsel$ (resp. $\gselp$, $\esel$, $\sel$) of degree $d$.

As we have remarked in \cite{Ragh20}, the conditions (G3) and (G4) impose a uniform bound of $1/2$ on both the archimedean and non-archimedean local parameters, perhaps making the class $\gselp$ a more aesthetically satisfying choice. By contrast, the Selberg class demands $\re(-\mu_j)\le 0$, but $\theta<1/2$. The condition (G3) ensures that there are no trivial zeros of $F$ on the critical line $\re(s)=1/2$.

For $F(s)\in \gselp$ and $S$ a finite set of primes, we define the partial $L$-function (or partial Dirichlet series) by 
\[
F_S(s)=\prod_{p\notin S}F_p(s).
\]
If $F(s)=\sum_{n=1}^{\infty}a_nn^{-s}$ and $\chi$ is a 
Dirichlet character, we define the twisted series
\begin{equation}\label{selchartwist}
F(s,\chi)=\sum_{n=1}^{\infty}\frac{\chi(n)a_n}{n^s}.
\end{equation}
We note that if $F(s)=F_1(s)F_2(s)$ is a product of two Dirichlet series,
\begin{equation}\label{twistprod}
F(s,\chi)=F_1(s,\chi)F_2(s,\chi).
\end{equation}
It is a conjecture of Selberg that the class $\sel$ is closed under twisting by 
Dirichlet characters. We expect that the same is true for the class $\gselp$ (in fact, one would conjecture
that $\gselp=\sel$). In what follows we may sometimes require $F(s)$ to satisfy the following property:
\begin{enumerate}
 \item[(G5)] For every non-principal Dirichlet character $\chi$, the series
$F(s,\chi)=\sum_{n=1}\frac{\chi(n)a_n}{n^s}$ extends to a meromorphic function on $\C$ and is holomorphic on the line $\re(s)=1$.
\end{enumerate}

We could have formulated our results somewhat more generally and avoided (to some extent) invoking (P4) (and (G4)), since many of the necessary inputs have been proved for the much larger class $\asel$. However, the sets $\esel_0$ and $\asel_0$ are quite large (and coincide), and this prevents a clean formulation of factorisation results. In contrast $\gselp_0=1$ (see Theorem \ref{gselpzero} of this paper), so we do not have to worry about degree zero elements.

The main (and conjecturally, the only) examples of Dirichlet series in the various classes above arise from automorphic $L$-functions. Let $\sigma$ be a unitary automorphic representation of $\glnak$, where $K$ is a number field, $\ak$ its ring of ad\`eles, and $m=[K:\Q]$ its degree. We will denote by $L(s,\sigma)$ the standard $L$-function associated to $\sigma$. As we will see in Theorem \ref{lautoing} of Subsection \ref{subsecgloball}, $L(s,\sigma)$ belongs to $\gselp$, and its degree is $mn$ (although Theorem \ref{lautoing} is stated only for $\sigma\in \autnall$, it is valid over number fields). This the main reason we choose to work in in $\gselp$. In contrast, since the Generalised Ramanujan Conjecture (at infinity) is very far from being proved for these $L$-functions, they are not known to belong to $\sel$. The only known examples are $L$-functions of Hecke characters of number fields, those arising in some way (e.g. as symmetric powers) from holomorphic cuspidal eigenforms (of any level) or Maass cuspidal eigenforms of small level, or those arising from sufficiently nice Galois representations.  We will discuss these  automorphic $L$-functions in greater detail in Section \ref{secautomorph}.

\section{An analogue of Booker's theorem for the class $\gselp$}\label{secbookanalog}

The aim of this section is to prove the following variant of the main theorem of \cite{Booker2003}.
\begin{theorem} \label{bookermodthm} Suppose that $F(s)\in \gsel_2$ and that it satisfies (G5).
Then $F(s,\chi)$ has no poles in $0<\re(s)\le 1$, if $\chi$ is not principal.
\end{theorem}
We first explain how Theorem \ref{primthm} differs from the various avatars already present in the literature. 
The main (and only) theorem of \cite{Booker2003} is the special case of Theorem \ref{bookermodthm} for $L$-functions of two-dimensional Artin representations, that is, two-dimensional complex representations of the absolute Galois group of $\Q$. These give rise to Dirichlet series of degree $2$ with an Euler product and satisfying the functional equation \eqref{fnaleqn} with
$G(s)=\Gamma_{\R}(s+a)^2$, with $a\in \{0,1\}$  or
$G(s)=\Gamma_{\C}(s)$.
For our purposes however, we need a theorem that allows the factors $G(s)$ to be much more general. In particular, they should be allowed to have gamma functions in the denominators (that is, we allow for factors with $r^{\prime}>0$). We also
cannot assume that the Dirichlet series have Euler products. The theorems of \cite{BoKr2014} for $G(s)=\Gamma_{\C}\left(s+\frac{k-1}{2}\right)$ ($k\in \N$), and \cite{NeOl2020} for $\Gamma(s)$ corresponds to the equation for classical Maass forms avoid the assumption of an Euler product for $F(s)$.
However, these two cases do not account even for all the factors $G(s)$ that arise from autmorphic representations of $\gltwo$, for instance the case when
$G(s)=\Gamma_{\R}(s+it)\Gamma_{\R}(s+1-it)$, let alone quotients of such functions. In any event, Theorem \ref{bookermodthm} allows us to uniformly treat not only all factors that arise from automorphic representations of $\gltwo$, but the completely general $L$-factors $G(s)$ of degree $2$ that arise in the class $\gsel$.

We should acknowledge that the theorems of \cite{BoKr2014} and \cite{NeOl2020} do not require the holomorphy hypothesis made in (G5) and are sharper since they require twisting by a much smaller (but still, infinite) set of characters. 
They also do not require the condition (G1), but only the weaker (P1). We have not attempted to modify those proofs to suit our more general setting: our arguments would have become considerably longer without allowing any new applications. More recently, the authors of \cite{BFL2022} treat the Weil functional equation and generalise certain aspects to the Selberg class. The proof of Theorem \ref{bookermodthm} relies only on the older methods of \cite{Booker2003}.

Finally, we remark that the condition that $F(s,\chi)$ be holomorphic on the line $s=1$ will easily be seen to be satisfied when $F=F_2$, where $F_2$ is as in Theorem \ref{zerothm}. This is because $F_2$ is the quotient of $L(s,\pi)$ and $L(s,\rho)$. It is known that $L(s,\pi\times\chi)$ is holomorphic on the line $\re(s)=1$ by a result of \cite{GoJa72} (see Theorem \ref{sigmaponeptwo} of this paper), and that $L(s,\rho\times\chi)$ is non-vanishing for $\re(s)=1$ by \cite{JaSh1976} (see Theorem \ref{nonvanishthm}), whence (G5) follows.

\begin{proof}[Proof of Theorem \ref{bookermodthm}]
Note that $\sum_{j=1}^r\lambda_j-\sum_{j^{\prime}}^{r^{\prime}}\lambda^{\prime}_{j^{\prime}}=1$, since $F(s)$ is in $\gsel_2$. We set 
\[
D=(2\pi)^{\frac{r-r^{\prime}-1}{2}},\,\, K^{\prime}=\prod_{j,j^{\prime}=1}^{r,r^{\prime}}\lambda_{j}^{\lambda_j}{\lambda^{\prime}_{j^{\prime}}}^{-\lambda^{\prime}_{j^{\prime}}}\,\,\text{and}\,\,
\mu=\sum_{j,j^{\prime}=1}^{r,r^{\prime}}\mu_j-\mu^{\prime}_{j^{\prime}}+\frac{r-r^{\prime}-1}{2}.
\]
It is a straightforward consequence of Stirling's formula that for any $k\ge 0$
\begin{equation}\label{stirlingcons}
G(s)=D{K^{\prime}}^s\Gamma(s+\mu)\left[1+\frac{c_1}{(s+\mu)}+\cdots +\frac{c_k}{(s+\mu)^k}+O\left(\frac{1}{s^{k+1}}\right)\right],
\end{equation}
for suitable constants $c_k$. 
Following the proof of Lemma 3 of \cite{Booker2003}, we use the functional equation for the gamma function (repeatedly) in \eqref{stirlingcons} to
obtain the following proposition.
\begin{proposition}\label{replacelemthree} For any $n\le m$, 
there are constants $b_1,\ldots b_n, A\in \C$ and $B>0$ such that 
\[
 G(s)(s-1/2)^{m}=AB^s\left[\sum_{k=n}^{m}b_k(2\pi)^{-s}\Gamma(s+k+\mu)]+(2\pi)^{-s}\Gamma(s+n+\mu)E_n(s)\right],
 \]
where $E_n(s)=O(1/s)$.
\end{proposition}
The proposition above is the analogue of Lemma 3 of \cite{Booker2003}. To prove Theorem \ref{bookermodthm} all we have to 
really do is to replace Lemma 3 in that paper with Proposition \ref{replacelemthree} and follow the remainder of the proof. Although the modifications
 required are routine, we give a very quick sketch of the proof for the sake of completeness.
We will assume that for some (non-principal) Dirichlet character $\chi$, $F(s,\chi)$ has a pole at $s=\beta_0$ with $0\le \re(\beta_0)\le 1$. Because of assumption (G5), such a
pole must necessarily lie in the critical strip $0<\re(s)<1$. By the character inversion formula, it follows this pole must be a pole of some additive twist $F(s,\alpha):=\sum_{n=1}^{\infty}a_ne^{2\pi in\alpha}n^{-s}$,
for some rational number $\alpha$.  Following \cite{Booker2003} we consider the expression
\begin{equation}\label{resterm}
 \sum_{\beta\in T}\res_{s=\beta}G(s)F(s)\left(\alpha Be^{i(\frac{\pi}{2}-\delta)}\right)^{\frac{1}{2}-s}(s-1/2)^{m},
\end{equation}
where $\beta$ runs over the set of poles $T$ of $G(s)F(s)$,
$0<\delta<\frac{\pi}{2}$, $B$ is the constant appearing in Proposition \ref{replacelemthree}, and $m$ is a positve integer to be chosen subsequently. Since we have assumed that $F(s)$ satisfies (P2), it has at most finitely many poles in $\C$. We know that $\Phi(s)=Q^sG(s)F(s)$ is holomorphic in $\re(s)>1$, since $G(s)$ is holomorphic in $\re(s)>1/2$, and $F(s)$ satisfies (G1), so is holomorphic in $\re(s)>1$. Because of the functional equation \eqref{fnaleqn}, $\Phi(s)$ is holomorphic in $\re(s)\le 0$. Since $G(s)$ has only finitely many poles in any strip of finite width, it has only finitely many poles in $0< \re(s)<1$. Thus $\Phi(s)$, and hence, $G(s)F(s)$ has only finitely many poles in $\C$, so the set $T$ is finite.
We will show that if $\re(\beta)>0$ for some $\beta\in T$, the sum \eqref{resterm} can be made arbitarily large 
as $\delta\to 0$. However, if $F(s)$ has only finitely many poles as we have assumed, \eqref{resterm} remains bounded as $\delta\to 0$, producing a contradiction. Thus, $F(s,\chi)$ cannot have poles in $0<\re(s)\le 1$ as desired.

Standard techniques involving shifting the line of integration, together with the functional equation \eqref{fnaleqn}, yield the following analogue of 
equation (3) of \cite{Booker2003}:
\begin{flalign}\label{three}
&\sum_{\beta\in T}\res_{s=\beta}G(s)F(s)\left(\alpha Be^{i(\frac{\pi}{2}-\delta)}\right)^{\frac{1}{2}-s}
\left(s-1/2\right)^{m}=\nonumber\\
&\frac{1}{2\pi i}\int_{(\sigma)}
\left[F(s)G(s)\left({\alpha}Be^{i(\frac{\pi}{2}-\delta)}\right)^{\frac{1}{2}-s}
-(-1)^m\omega\tilde{G}(s)\tilde{F}(s)\left({(\alpha}BQ^2)^{-1}e^{-i(\frac{\pi}{2}-\delta)}\right)^{\frac{1}{2}-s}\right]\nonumber\\
&\hskip 7.5cm \left(s-1/2\right)^{m} ds,
\end{flalign}
where $\int_{(\sigma)}$ indicates the integral on the line $\reals=\sigma$, for 
some $\sigma>1$.
Using Proposition \ref{replacelemthree} we see that the equation above can be rewritten as
\begin{flalign}\label{eighteen}
&\sum_{s=\beta}
 G(s)F(s)\left(\alpha Be^{i(\frac{\pi}{2}-\delta)}\right)^{\frac{1}{2}-s}
(s-1/2)^{m}\nonumber\\
 &=\frac{1}{2\pi i}\int_{(\sigma)}
 F(s)\left(\alpha Be^{i(\frac{\pi}{2}-\delta)}\right)^{\frac{1}{2}-s}
\left[ AB^s\sum_{k=n}^{m}b_k(2\pi)^{-s}\Gamma(s+k+\mu)\right]\nonumber\\
&-(-1)^m\omega\tilde{F}(s)\left({(\alpha}BQ^2)^{-1}e^{-i(\frac{\pi}{2}-\delta)}\right)^{\frac{1}{2}-s}
\left[ AB^s\sum_{k=n}^{m}b_k(2\pi)^{-s}\Gamma(s+k+\bar{\mu})\right]\nonumber\\
&+F(s)\left(\alpha Be^{i(\frac{\pi}{2}-\delta)}\right)^{\frac{1}{2}-s}
(2\pi)^{-s}\Gamma(s+n+\mu)E_n(s)\nonumber\\
&-(-1)^m\omega\tilde{F}(s)\left({(\alpha}BQ^2)^{-1}e^{-i(\frac{\pi}{2}-\delta)}\right)^{\frac{1}{2}-s}(2\pi)^{-s}\Gamma(s+n+\bar{\mu})\tilde{E_n}(s)ds
\end{flalign}
Let $n=[-1-\re(\mu)]$, and consider the last two terms in the right hand side of \eqref{eighteen} above.
If we shift the contour to 
$\sigma=-n-\re(\mu)+\Delta$ for $0<\Delta<1/2$, we note that 
$\re(\sigma)>1$, so we are in the domain of (absolute) convergence of $F(s)$ and $\tilde{F}(s)$.
It follows that the integrands in these terms are $O(1/|s|^{3/2-\Delta})$ independent of $\delta$, so the integrals are $O(1)$.

We now treat the first two terms on the right hand side of \eqref{eighteen}.
Using Lemma 2 of \cite{Booker2003} when $\sigma=k+\mu$ (this lemma just uses the fact that 
the function $e^{-z}$ is the inverse Mellin transform of $\Gamma(s)$), we see that the right-hand side in \eqref{three} is $O(1)$ plus a linear combination of terms of the form
\begin{flalign}\label{postlemtwo}
&\frac{1}{2\pi i}\int_{(\sigma)}\bigg[B\Gamma\left(s+k+\mu\right)F(s,\alpha)\alpha^{\frac{1}{2}-s}
e^{\frac{i\delta}{2}(s-k-\mu-1)+\frac{i\pi}{2}(k+\mu+\frac{1}{2})}-\nonumber\\
&(-1)^{m}\omega \Gamma(s+k+\bar{\mu})\tilde{F}(s,(BQ^2\alpha)^{-1})\left(( B^2Q^2\alpha)^{-1}\right)^{\frac{1}{2}-s}e^{-\frac{i\delta}{2}(s-k-\mu-1)+\frac{i\pi}{2}(k+\mu+\frac{1}{2})}\bigg]\nonumber\\
&\hskip 7cm\times\left(2\sin\frac{\delta}{2}\right)^{-(s+k+\mu)}(2\pi)^{-s}ds,
\end{flalign}
where $n\le k\le m$. For $k$ such that $k+\re(\mu)<-1$, we shift the line of 
integration to $1+\sigma_0$, where $\sigma_0>0$ is such that
$-k-\re(\mu)>1+\sigma_0$ for all such $k\ge n$. This will give rise to residues arising from the poles of $\Gamma\left(s+k+\mu\right)$, but these will be independent of $\delta$. It follows that their contribution will be $O(1)$. The integrals on the line $\sigma=1+\sigma_0$ will likewise be $O(1)$.

Modifying \cite{Booker2003} once again, we consider the $k=m$ term and set 
\begin{flalign}
&f(s,\delta)=B\Gamma\left(s+m+\mu\right)F(s,\alpha)\alpha^{\frac{1}{2}-s}
e^{\frac{i\delta}{2}(s-m-\mu-1)+\frac{i\pi}{2}(m+\mu+\frac{1}{2})}-\nonumber\\
&(-1)^{m}\omega \Gamma(s+m+\bar{\mu})\tilde{F}(s,(BQ^2\alpha)^{-1})\left(( BQ^2\alpha)^{-1}\right)^{\frac{1}{2}-s}e^{-\frac{i\delta}{2}(s-m-\mu-1)+\frac{i\pi}{2}(m+\mu+\frac{1}{2})}.
\end{flalign}
Note that 
\begin{flalign}
f(s,0)&=B\Gamma\left(s+m+\mu\right)F(s,\alpha)\alpha^{\frac{1}{2}-s}
e^{\frac{i\pi}{2}(m+\mu+\frac{1}{2})}\nonumber\\
&-(-1)^{m}\omega \Gamma(s+m+\bar{\mu})\tilde{F}(s,(BQ^2\alpha)^{-1})\left(( BQ^2\alpha)^{-1}\right)^{\frac{1}{2}-s}e^{\frac{i\pi}{2}(m+\mu+\frac{1}{2})}\nonumber
\end{flalign}
Since $f(s,\delta)$ is holomorphic as a function of $\delta$, we may 
write $f(s,\delta)=f(s,0)+f_1(s,\delta)\delta$, where $f_1(s,\delta)$ is holomorphic in $\delta$ in a neighbourhood of $0$.
Thus the term $k=m$ in the expression \eqref{postlemtwo} has the form
\[
\frac{1}{2\pi i}\int_{(\sigma)}[f(s,0)+f_1(s,\delta)\delta] \left(2\sin\frac{\delta}{2}\right)^{-(s+m+\mu)}(2\pi)^{-s}ds.
\]
Recall that we have assumed that $F(s,\alpha)$ has a pole at $s=\beta_0$ with
$0<\re(\beta_0)<1$.
We now have two possible cases. Suppose first that $f(s,0)$ has a pole at $s=\beta_0$. Then by Lemma 4 of \cite{Booker2003}, we see that
\[
\int_{(\sigma)}f(s,0)(2\sin \delta/2)^{-s-m-\mu}ds=\Omega_{\varepsilon}(\delta^{-\beta_0-m-\re(\mu)+\varepsilon}).
\]
The integral $\int_{(\sigma)}f_1(s,\delta)\delta(2\sin \delta/2)^{-s-m-\mu}ds$ is obviously $O(\delta^{-m-\re(\mu)-\varepsilon})$ if we
take $\sigma=1+\varepsilon$. Similarly, we see that the integrals in the expression \eqref{postlemtwo} with $k<m$ are also $O(\delta^{-m-\re(\mu)-\varepsilon})$. Since $\re(\beta_0)>0$, we see that for
$\varepsilon$ small enough, the first integral dominates the second
for infinitely many $\delta$ as $\delta\to 0$ for $k=m$, as also the integrals when $k<m$, contradicting the assertion that the expression \eqref{resterm} is $O(1)$ as $\delta\to 0$.

If $f(s,0)$ is entire, we can simply follow the last part of the argument in \cite{Booker2003} (see pages 1096-1097) to get a similar contradiction and thus obtain our theorem.
\end{proof}
\begin{corollary}\label{phischientire} Under the assumptions of the theorem,
if $\chi\ne 1$ is a primitive character, $\Phi(s,\chi)=G(s,\chi)F(s,\chi)$ is entire.
\end{corollary}
\begin{proof} 
We note $G(s,\chi)$ is holomorphic in $\re(s)\ge 1/2$ by (G3). Hence, by (G1) and (G5), $\Phi(s,\chi)$ is holomorphic in $\re(s)\ge 1$.

Suppose that $\Phi(s,\chi)$ has a pole at $s=\beta$ with 
$0<\re(\beta)\le 1$. By the theorem above, we know that $F(s,\chi)$ has no poles in the strip $0<\re(s)\le 1$. It follows that $\beta$ is a pole of $G(s,\chi)$, so we must have $0<\re(\beta)<1/2$, and we further know that $(1-\beta)$ is not a pole of $G(s,\chi)$. Because of the functional equation, we know that $1-\beta$ is a pole of $\Phi(s,\chi)$, and hence a pole of $F(s,\chi)$. But $1/2<\re(1-\beta)\le 1$, which contradicts the theorem.

Suppose $\Phi(s,\chi)$ has a pole at $s=\beta$ and $\re(\beta)\le 0$. Then $\re(1-\beta)\ge 1$, so $1-\beta$ cannot be a pole of $\Phi(s,\chi)$, since
it is not a pole of $F(s,\chi)$ by (G1) and (G5), and not a pole of $G(s,\chi)$ by (G3). It follows from \ref{fnaleqn} that $\beta$ cannot be a pole of $\Phi(s,\chi)$.
\end{proof}

\section{Local and global $L$-functions of automorphic representations}\label{secautomorph}
We will quickly recall some more facts about automorphic $L$-functions. We will assume that we are working over $\Q$, though most of the theorems we state below are valid over arbitrary number fields. If $\sigma\in \autnall$, it is known that $\sigma=\otimes^{\prime}_v\sigma_v$, the restricted direct product of irreducible admissible representations
of $\glnqv$, where $v$ runs over the set of (finite and infinite) places of $\Q$. 
We let $\chi=\otimes_v^{\prime}\chi_v$ be a Dirichlet character viewed as an id\`ele class character. We also fix a global additive character $\psi=\otimes_v^{\prime}\psi_v$ of $\aq$ which is trivial on $\Q$.

\subsection{Non-archimedean local factors}\label{subsecnonarchl}
By the work of Godement-Jacquet in \cite{GoJa72} we can attach a local $L$-function $L(s,\sigma_p)$ to the representation $\sigma_p$. 
Explicitly, it has the form
\begin{equation}\label{nonarchlfn}
L(s,\sigma_p)=\prod_{i=1}^n\frac{1}{(1-\alpha_i(p)p^{-s})}
\end{equation}
for complex numbers $\alpha_i(p)$, $1\le i\le n$. For all but finitely many
primes $p$ we have $\alpha_i(p)\ne 0$, for $1\le i\le n$. These are called the unramified places of $\sigma$ or the places where $\sigma_v$ is unramified.
The remaining finitely many places are said to be the places where $\sigma$ is ramified. We denote this set of places by $S_{\sigma}$. If $\chi=\otimes^{\prime}_p\chi_p$ is a primitive Dirichlet character, we may attach the $L$-factor to the twist $\sigma_p\times\chi_p$ of $\sigma_p$ by $\chi_p$. 
Whenever $p\notin S_{\sigma}\cup S_{\chi}$, the twisted $L$-function has the form
\[
L(s,\sigma_p\times \chi_p)=\prod_{i=1}^n\frac{1}{(1-\chi(p)\alpha_i(p)p^{-s})}.
\]
Recall that we may also associate an $\varepsilon$-factor 
\[
\varepsilon(s,\sigma_p\times\chi_p,\psi_p)=r(\sigma_p\times\chi_p,\psi_p)K(\sigma_p\times\chi_p,\psi_p)^{1/2-s}
\]
to each pair $(\sigma_p,\chi_p)$. Here $r(\sigma_p\times\chi_p,\psi_p)$ is a complex number of absolute value $1$, and $K(\sigma_p\times\chi_p,\psi_p)=p^{n_{p,\chi}}$, for some integer $n_{p,\chi}\ge 0$ which depends on $\psi_p$. It is a fact that the local $\varepsilon$-factors are identically $1$ at the primes $p$ where both $\sigma_p$ and $\chi_p$ are unramified and, in fact, $r(\sigma_p\times\chi_p,\psi_p)=1=K(\sigma_p\times\chi_p,\psi_p)$ at these places.

We will really need to know the $\varepsilon$-factor only for characters. 
For a character $\chi_p$ of conductor $P=p^e$, the $\varepsilon$-factor is given by 
\begin{equation}\label{epschip}
\varepsilon(s,\chi_p)=(-1)^{\epsilon_{\chi_p}}\tau(\chi)P^{-s},
\end{equation}
where $\tau(\chi)$ is the Gauss sum associated to the character $\chi_p$ and the additive character $\psi_p$. Thus
$r(\chi_p,\psi_p)=\frac{\tau(\chi)}{P^{1/2}}$ and $K(\chi_p,\psi_p)=P$.

\subsection{Archimedean local factors}\label{susbsetarchl}
We may also define $L$-factors at the infinite place. They have the form
\begin{equation}\label{archlfn}
L(s,\sigma_{\infty})=\prod_{l=1}^{r_1}\Gamma_{\R}(s+\epsilon_l+\nu_l)\prod_{l=r_1+1}^{r_1+r_2}
\Gamma_{\C}\left(s+\frac{k_l-1}{2}+\nu_l\right),
\end{equation}
where 
\begin{equation}\label{archcases}
\Gamma_{\R}(s)=\pi^{-\frac{s}{2}}\Gamma\left(\frac{s}{2}\right)
~\text{and}~\Gamma_{\C}(s)=2(2\pi)^{-s}\Gamma(s).
\end{equation}
Here, the $\nu_l:=\nu_l(\sigma_{\infty})$ are complex numbers and $\epsilon_l:=\epsilon_l(\sigma_{\infty})$ is $1$ or $0$ according to whether or not the 
sign representation of $\R^{*}$ arises in the relevant representation of the Weil group. The $k_l:=k_l(\sigma_{\infty})\ge 1$ are integers for $r_1+1\le l\le r_1+r_2$ and $r_1+2r_2=n$. The condition that $\sigma$ be unitary yields
\begin{equation}\label{archunit}
\sum_{l=1}^{r_1}\re(\nu_l)+\sum_{l=r_1+1}^{r_1+r_2}2\re(\nu_l)=0.
\end{equation}
Assume now that $\chi$ is a Dirchlet character, so $\chi_{\infty}$ is either the trivial or sign character.
For the representations $\sigma_{\infty}\times\chi_{\infty}$, the corresponding
$L$-factors are
\begin{equation}\label{twistedarchlfn}
 L(s,\sigma_{\infty}\times\chi_{\infty})=\prod_{l=1}^{r_1}\Gamma_{\R}(s+\epsilon_{l,\chi}+\nu_l)\prod_{l=r_1+1}^{r_1+r_2}
\Gamma_{\C}\left(s+\frac{k_l-1}{2}+\nu_l\right),
 \end{equation}
where $\epsilon_{l,\chi}=\epsilon_l+\epsilon_{\chi}\pmod 2$, and 
$\epsilon_{\chi}=0$ if $\chi(-1)=1$, and $\epsilon_{\chi}=-1$ otherwise.

As in the non-archimedean case, we may associate a factor 
$\varepsilon(s,\sigma_{\infty}\times\chi_{\infty},\psi_{\infty})$ to the representation $\sigma_{\infty}\times\chi_{\infty}$. It has the form $r(\sigma_{\infty}\times\chi_{\infty},\psi_{\infty})K(\sigma_{\infty}\times\chi_{\infty},\psi_{\infty})^{1/2-s}$ with
$r(\sigma_{\infty}\times\chi_{\infty},\psi_{\infty})$ a complex number of absolute value $1$ and 
$K(\sigma_{\infty}\times\chi_{\infty},\psi_{\infty})>0$.

For future reference, we write out the archimedean factors more explicitly for a Dirichlet character $\chi$ and $\tau\in \auttwoall$. We will take our additive character to be $\psi_{\infty}(x):=e^{2\pi ix}$.

For a Drichlet character $\chi$
\begin{equation}\label{ginftyone}
L(s,\chi_{\infty})=\Gamma_{\R}(s+\epsilon_{\chi}),
\end{equation}
where $\epsilon_{\chi}\in \{0,1\}$ is defined by the equation $\chi(-1)=(-1)^{\epsilon_{\chi}}$. The $\varepsilon$-factor is given by
\begin{equation}\label{epschiinfty}
\varepsilon(s,\chi_{\infty})=i^{\epsilon_{\chi}}.
\end{equation}
For $\tau\in \auttwoall$, and $\chi$ as above, we have
\begin{equation}\label{ginftytwo}
L(s,\tau_{\infty}\times\chi_{\infty})=\begin{cases}\Gamma_{\R}(s+\epsilon_{1,\chi}+\nu+ib_1)
\Gamma_{\R}(s+\epsilon_{2,\chi}-\nu+ib_2)\,\,\text{if $r_1=2$, or}\\
\Gamma_{\C}\left(s+\frac{k-1}{2}+ib_3\right)\,\,\text{if $r_1=0$ and $r_2=1$},
\end{cases}
\end{equation}
where $\epsilon_{j,\chi}=\epsilon_{j}+\epsilon_{\chi}\pmod 2$ and
$\epsilon_{j}\in \{0,1\}$, $j=1,2$, and $\nu,b_1,b_2, b_3\in \R$, 
and $k\in \N$. The $\varepsilon$-factor is given by
\begin{equation}\label{epsiloninftytwo}
\varepsilon(s,\tau_{\infty}\times\chi_{\infty})=\begin{cases}i^{\varepsilon_1+\epsilon_2}\,\,\text{if $r_1=2$, or}\\
i^k\,\,\text{if $r_1=0$ and $r_2=1$}.
                                                 \end{cases}
\end{equation}
We note that the archimedean $\varepsilon$-factor above is independent of 
the character $\chi_{\infty}$ by which we are twisting.

\begin{theorem} \label{jslocboundsthm} Let $\sigma_v$ be an irreducible unitary representation of $\glnqv$.
\begin{enumerate}
\item If $v=p$ is finite, then
\begin{equation}\label{nonarchbounds}
p^{-1/2}<|\alpha_i(p)|<p^{1/2}
\end{equation}
for all $1\le i\le n$.
\item If $v=\infty$, then 
\begin{equation}\label{archbounds}
-1/2<\re(\nu_l)<1/2
\end{equation}
for all $1\le \nu_l\le r_1+r_2$.
\end{enumerate}
\end{theorem}

\subsection{The multiplicativity of $L$- and $\varepsilon$-factors}
Let $\sigma_v$ be an irreducible admissible representation of $\glnqv$, and suppose that 
\[
\sigma_v=\sigma_{1,v}\boxplus \sigma_{2,v}\boxplus \cdots \boxplus \sigma_{k,v}.
\]
Recall that this means that there are (quasi-)cuspidal representations 
$\sigma_{i,v}$ of $\glniqv$, $1\le i\le k$, with $n_1+\cdots
+n_k=n$, such that $\sigma_v$ appears as the (unique)
irreducible quotient of the parabolically induced representation
$I_{\sigma_{1,v},\ldots ,\sigma_{k,v}}$
arranged so that $\re(s_i)\ge \re(s_{i+1})$ for $1\le i\le k-1$,
where $\omega_{\sigma_{i,v}}(z)=|z|^{s_1}$ is the central character of $\sigma_{i,v}$. It is a theorem of \cite{JPSS83} that the $L$-factors and $\varepsilon$-factors associated to $\sigma_v$ are multiplicative: 
\begin{equation}\label{lfnofindrep} 
L(s,\sigma_v)=\prod_{i=1}^kL(s,\sigma_{i,v})\quad\text{and}\quad 
\varepsilon(s,\sigma_v)=\prod_{i=1}^k\varepsilon(s,\sigma_{i,v}).
\end{equation}

\subsection{The global theory}\label{subsecgloball}
Let $\sigma=\otimes_{v}^{\prime}\sigma_v$ be an irreducible unitary representation of $\gln$ with a unitary automorphic central character $\omega_{\sigma}$ (any $\sigma\in \temprep$ is of this form).
Let $S$ be a finite subset of the places of $\Q$. We may define the partial $L$-functions
\[
L_S(s,\sigma\times \chi):=\prod_{p\notin S}L(s,\sigma_p\times\chi_p),
\]
as well as the Dirichlet series and (completed) $L$-function associated to $\sigma\times \chi$ as
\[
L(s,\sigma\times\chi)=\prod_{p<\infty}L(s,\sigma_p\times\chi_p)\quad\text{and}\quad\Lambda(s,\sigma\times\chi)=\prod_{v}L(s,\sigma_v\times\chi_v)
\]
respectively. If $\chi$ is a Dirichlet character, we get the corresponding $L$-functions for $\sigma\times\chi$.
Likewise, we have the global $\varepsilon$-factors
\[
\varepsilon(s,\sigma\times\chi)=\prod_v\varepsilon(s,\sigma_v\times\chi_v)
\]
attached to $\sigma\times\chi$ (recall that all but finitely many factors in the product above are $1$). While the local $\varepsilon$-factors depend
on the choices of local additive characters of $\Q_p$, the global $\varepsilon$-factors are independent of the choice $\psi$ of a global additive character.
We see that 
\[
\varepsilon(s,\sigma\times\chi)=r(\sigma\times\chi)K(\sigma\times\chi)^{1/2-s},
\]
where
\[
r(\sigma\times\chi)=\prod_vr(\sigma_v\times\chi_v)\quad\text{and}\quad K(\sigma\times\chi)=\prod_{v}K(\sigma_{v}\times\chi_{v}).
\]

The theorem below is a collection of results from \cite{GoJa72}.
\begin{theorem}\label{sigmaponeptwo} Let $\sigma\in \autn$
\begin{enumerate}
 \item The function $L(s,\sigma)$ satisfies (P1) and has at most finitely many poles. Further, if $\sigma\in \autn$, it is entire.
 \item For every primitive Dirichlet character $\chi$, we have the functional equation
 \begin{equation}\label{autlfneqn}
\Lambda(s,\sigma\times\chi)=\varepsilon(s,\sigma\times\chi)\Lambda(1-s,\tilde{\sigma}\times\chi^{-1}),
\end{equation}
where $\tilde{\sigma}$ denotes the contragredient of $\sigma$.
\end{enumerate}
\end{theorem}
Since $\sigma$ is unitary 
\[
L(1-s,\tilde{\sigma}_v\times\chi_v^{-1})=\tilde{L}(1-s,\sigma_v\times\chi_v).
\]
for every place $v$. It follows that $L(s,\sigma\times\chi)$ satisfies a functional equation of the form 
given in \eqref{fnaleqn}, thereby establishing that $L(s,\sigma\times\chi)$ satisfies (P3). 

The growth condition in (P2) was proved  for the standard $L$-function $L(s,\sigma)$, $\sigma\in \temprepn$ in 
\cite{JPSS792}, and for a larger class of $L$-functions associated to automorphic representations by  Gelbart and Shahidi in \cite{GeSh01}. In conjunction with the first part of Theorem \ref{sigmaponeptwo}, we see that $L(s,\sigma)$ satisfies (P2). 

The condition (P4) is satisfied by $L(s,\sigma)$ once (P1) is known. We thus see that $L(s,\sigma)\in \aselp$ for $\sigma\in \autn$. By using \eqref{lfnofindrep}, we see that $L(s,\sigma)\in \aselp$ for all $\sigma\in \temprep_n$.

If $\sigma\in \autn$, it follows from the results of Jacquet and Shalika in \cite{JaSh812} and \cite{JaSh811} that $L(s,\sigma)$ satisfies (G1). 
Since cuspidal automorphic representations are globally generic, every local constituent $\sigma_v$ of $\sigma$ is unitary and 
generic. For such local constituents, Serre observed  (the relevant letter is reproduced in \cite{BlBr2013}) that the work of Jacquet and Shalika in \cite{JaSh812,JaSh811} together with Theorem 4.1 of Chandrasekharan and Narasimhan in \cite{KCRN62} yields improvements on Theorem \ref{jslocboundsthm}, while the archimedean bounds were first proved in \cite{LRS1999}.\begin{theorem} \label{jsboundsthm} Let $\sigma=\otimes_v^{\prime}\sigma_v$ be an unitary cuspidal automorphic representation of $\gln$.
\begin{enumerate}
\item If $v=p$ is finite, then
\begin{equation}\label{nonarchbounds}
p^{-1/2+1/(n^2+1)}<|\alpha_i(p)|<p^{1/2-1/(n^2+1)}
\end{equation}
for all $1\le i\le n$.
\item If $v=\infty$, then 
\begin{equation}\label{archbounds}
-1/2+1/(n^2+1)<\re(\nu_l)<1/2-1/(n^2+1)
\end{equation}
for all $1\le \nu_l\le r_1+r_2$.
\end{enumerate}
\end{theorem}
\begin{remark} For archimedean $v$, Theorem \ref{jslocboundsthm} is adequate for our purposes.
\end{remark}
\begin{remark} Over number fields, the corresponding results were proved in \cite{LRS1995} and \cite{LRS1999}.
\end{remark}
\begin{remark} \label{ramanujanrmk} If $\sigma=\otimes^{\prime}_v\sigma_v\in \autn$, the statements $|\alpha_i(p)|=1$, $1\le i\le n$ for all $p$, and ${\re(\nu_l)}=0$,
$1\le l\le r_1+r_2$ for $v=\infty$ are equivalent to the Generalised Ramanujan Conjecture for $\sigma$. The conjecture has been established (by Deligne) when $\sigma\in \auttwo$, with $\sigma_{\infty}$ a (limit of) discrete series representation. It is also known to be true for a large class of Galois representations, namely those associated to regular algebraic, cuspidal automorphic representations.
\end{remark}
Theorem \ref{jsboundsthm} shows that $L(s,\sigma)$ satisfies (G3) and (G4), so $L(s,\sigma)\in \gselp_n$. If $\sigma\in 
\temprepn$, \eqref{lfnofindrep} together with Theorem \ref{jsboundsthm} shows that $L(s,\sigma)\in \gselp_n$.
We summarise our discussion above as the following theorem.

\begin{theorem}\label{lautoing} For $\sigma\in \temprepn$, 
$L(s,\sigma)\in \gselp$. If $\sigma$ further satisfies the Generalised Ramanujan Conjecture at infinity, $L(s,\sigma)\in \sel$.
\end{theorem}

We need one further input from the theory of automorphic $L$-functions proved by Jacquet and Shalika in \cite{JaSh1976}) for all $\sigma\in \aut$ from which the result for $\sigma\in \temprep$ below follows immediately using
\eqref{lfnofindrep}.
\begin{theorem} \label{nonvanishthm} If $\sigma\in \temprepn$, 
$L(1+it,\sigma)\ne 0$ for all $t\in \R$.
\end{theorem}

If $\pi_1,\pi_2\in \autn$, we have analogous statements for the tensor product (or Rankin-Selberg) $L$-functions
$L(s,\pi_1\times\pi_2)$. We can summarise the results of a series of papers due to Jacquet and Shalika 
\cite{JaSh812,JaSh811}, Jacquet, Piatetski-Shapiro and Shalika \cite{JPSS83}, Moeglin and Waldspurger \cite{MoWa89} 
and Shahidi \cite{Shahidi81,Shahidi88} and Gelbart and Shahidi \cite{GeSh01} in the following theorem.
\begin{theorem}\label{rsthm}  Let $\pi_1,\pi_2\in \autn$ and let $\chi$ be a primitive Dirichlet character. Then $L(s,\pi_1\times\pi_2\times\chi)$ 
satisfies (G1), (P2) (with $P(s)$ trivial or $P(s)=s-1$), a functional equation of the form \eqref{autlfneqn}, 
(P4) and (G5).
\end{theorem}
\begin{remark}\label{rsntrem} Under some natural additional hypotheses on $\pi_1$ or $\pi_2$, $L(s,\pi_1\times\pi_2\times\chi)$ will also satisfy (G3) and (G4). This happens, for instance, if one of $\pi_1$ and $\pi_2$ is tempered at each place or if both $\pi_1$ and $\pi_2$ satisfy \eqref{nonarchbounds} and \eqref{archbounds} with $1/2-1/(n^2+1)$ replaced by $1/4-\delta$ for some $\delta$ (these are sometimes called nearly tempered representations). If $\pi_1,\pi_2\in \autthree$ are symmetric square lifts of $\tau_1,\tau_2\in \auttwo$ respectively, their local components are known to be nearly tempered by \cite{Kim03}.
\end{remark}
\begin{remark} \label{exttwontrem} The analogue of Theorem \ref{rsthm} holds for the exterior square and symmetric square $L$-functions 
of $\pi\in \autn$. These results are due to Shahidi \cite{Shahidi90,Shahidi97}, Kim \cite{Kim99} and Takeda \cite{Takeda2015}. If (G3) and (G4) are to hold, we require that one of the representations be tempered, or that both be nearly tempered.
\end{remark}

\section{A $\gltwog$-type functional equation for $F_2(s)$}\label{redtogltwo}

In this section we will formulate our results for the $L$-functions of a larger class of representations of $\gln$. These will encompass Artin $L$-functions and other $L$-functions that are not (yet) known to be automorphic like the tensor product, symmetric square and exterior square $L$-functions.

We will make the following assumptions about $\pi$ and $\rho$ throughout this section.
Let $\pi=\otimes_v^{\prime}\pi_v$ and $\rho=\otimes_v^{\prime}\rho_v$ be irreducible admissible representations of $\gln$ and $\glntwo$. As before, we set $F_2(s):=L(s,\pi)/L(s,\rho)$ and $F_2(s,\chi):=L(s,\pi\times\chi)/L(s,\rho\times\chi)$ for a Dirichlet character $\chi$. Assume that $\pi_{\infty}$ and $\rho_{\infty}$ are unitary, so $F_2(s,\chi)$ satisfies (G3) by Theorem \ref{jslocboundsthm}, for any primitive Dirichlet character $\chi$.
Under suitable additional hypotheses on $F_2$, we will show that for primitive characters $\chi$, $F_2(s,\chi)$ satisfies a functional equation with an archimedean factor differing from that of the archimedean factor of an automorphic $L$-function of $\gltwo$ by a factor of a rational function. We make this precise below.

For any primitive Dirichlet character $\chi$, let us set
\begin{equation}\label{gtwodefn}
G_2(s,\chi):=G_2(s,\chi_{\infty})=L(s,\pi_{\infty}\times\chi_{\infty})/L(s,\rho_{\infty}\times\chi_{\infty})
\end{equation}
and
\begin{equation}\label{etwodefn}
E_2(s,\chi)=\varepsilon(s,\pi\times\chi)/\varepsilon(s,\rho\times\chi).
\end{equation}
When $\chi_{\infty}$ is trivial, we write $G_2(s)$ for $G_2(s,\chi_{\infty})$. Indeed, since we will largely deal with the archimedean component $\chi_{\infty}$ of $\chi$ in the rest of this proof, we will often refer to both the Dirichlet character $\chi$ and its archimedean component $\chi_{\infty}$ as $\chi$ with a view to keeping the notation lighter. From the context it will be clear whether we mean the local or the global character.

Using \eqref{twistedarchlfn} we have
\begin{equation}\label{twistedquotg}
G_2(s,\chi)=\frac{\prod_{p=1}^{r_1}\Gamma_{\R}\left(s+\epsilon_{p,\chi}+\nu_p\right)\prod_{p=r_1+1}^{r_1+r_2}\Gamma_{\C}\left(s+\frac{k_p-1}{2}+\nu_p\right)}
{\prod_{l=1}^{t_1}\Gamma_{\R}\left(s+\epsilon^{\prime}_{l,\chi}+\nu^{\prime}_l\right)\prod_{l=t_1+1}^{t_1+t_2}\Gamma_{\C}\left(s+\frac{k^{\prime}_l-1}{2}+\nu^{\prime}_l\right)},
\end{equation}
for suitable $\epsilon_{p,\chi},\epsilon_{l,\chi}^{\prime}\in \{0,1\}$ and $\nu_p,\nu_l^{\prime}\in \C$, and where $r_1+2r_2=n$ and $t_1+2t_2=n-2$. 

\begin{lemma}\label{gtwochifin} With notation as above, if $G_2(s)$ has a finite number of zeros, then so does $G_2(s,\chi)$, where $\chi_{\infty}$ is the sign character.
 \end{lemma}
 \begin{proof} 
We consider the gamma functions that appear as factors on the right hand side of \eqref{twistedquotg}.
By assumption, we see that for the trivial archimedean character, all but finitely many poles of the factor $\Gamma_{\R}(s+\epsilon^{\prime}_{l}+\nu_{l}^{\prime})$, $1\le l\le t_1$ must be poles of $L(s,\pi_{\infty})$.
It follows that there must exist $p=p(l)$, $1\le p\le r_1+r_2$. We analyse the various possibilities below, keeping in mind that (G3) ensures that $-1/2<\re(\nu_p),\re(\nu_{l}^{\prime})<1/2$, and hence, that
$-1/2<\epsilon_{p,\chi}+\re(\nu_p),\epsilon^{\prime}_{l,\chi}+\re(\nu_{l}^{\prime})<3/2$, $1\le p\le r_1+r_2$, $1\le l\le t_1+t_2$.

\begin{enumerate}
\item[Case 1.] Let $1\le l\le t_1$.
Suppose there is a $p=p(l)$, $1\le p= p(l)\le r_1$, such that infinitely many poles of $\Gamma_{\R}(s+\epsilon_p+\nu_p)$
are also poles of $\Gamma_{\R}(s+\epsilon^{\prime}_{l}+\nu_{l}^{\prime})$. It follows that $\epsilon_p+\nu_p=\epsilon_l+\nu^{\prime}_l+2m$ for some $m\in \Z$. Since $-1/2<\epsilon_p+\re(\nu_p),\epsilon^{\prime}_{l}+\re(\nu_{l}^{\prime})<3/2$, we must have $m=0$ and 
$\epsilon_p+\nu_p=\epsilon_l+\nu^{\prime}_l$. Thus, the factors  $\Gamma_{\R}(s+\epsilon_p+\nu_p)$ and $\Gamma_{\R}(s+\epsilon^{\prime}_l+\nu_{l}^{\prime})$ are identical so we may simply cancel them out. Moreover, since
$|\re(\nu_p)|,|\re(\nu_{l}^{\prime})|<1/2$, we must have $\epsilon_p=\epsilon^{\prime}_{l}$ and $\nu_p=\nu_{l}^{\prime}$.
It follows that 
the two factors $\Gamma_{\R}(s+\epsilon_{u,\chi}+\nu_{p(l)})$ and $\Gamma_{\R}(s+\epsilon^{\prime}_{m,\chi}+\nu_{l}^{\prime})$, which occur in the numerator and denominator of $G_2(s,\chi)$ respectively, are also identical and cancel each other out.
\item[Case 2.] Let $1\le l\le t_1$. If no $p=p(l)$, $1\le p(l)\le r_1$, as in Case 1 exists, there must be
a $p=p(l)$, $r_1+1\le p(l)\le r_2$, such that $\Gamma_{\C}\left(s+\frac{k_p-1}{2}+\nu_p\right)$ and $\Gamma_{\R}(s+\epsilon^{\prime}_{l}+\nu_{l}^{\prime})$ have
infinitely many poles in common. Let $m=[\frac{k_p-1}{2}+\nu_p]$ and write $\frac{k_p-1}{2}+\nu_p=m+\nu^{\prime\prime}$.
Note that $\re(\nu^{\prime\prime})<1$.
From the functional equation for the gamma function,
\[
\Gamma_{\C}\left(s+\frac{k_p-1}{2}+\nu_p\right)=\Gamma_{\C}\left(s+m+\nu^{\prime\prime}\right)
=P_l(s)\Gamma_{\C}\left(s+\nu^{\prime\prime}\right),
\]
for a polynomial $P_l(s)$ of degree $m=m(l)$ with zeros in the half-plane $\re(s)<1/2$. Using the duplication formula for the gamma function and comparing poles, we see that either 
$\nu^{\prime\prime}=\epsilon^{\prime}_{l}+\nu_{l}^{\prime}$, or  
$\nu^{\prime\prime}=\epsilon^{\prime\prime}_{l}+\nu_{l}^{\prime}$, for $\epsilon^{\prime\prime}_l=\epsilon^{\prime}_l
+1\pmod 2$. In either case, we see that 
\[
\frac{\Gamma_{\C}\left(s+\frac{k_p-1}{2}+\nu_p\right)}{\Gamma_{\R}(s+\epsilon^{\prime}_{l}+\nu_{l}^{\prime})}=P_l(s)\Gamma_{\R}(s+\epsilon^{\prime\prime}_{l}+\nu_{l}^{\prime}),
\]
If $\chi=\chi_{\infty}$ is the sign character, $\epsilon^{\prime}_{l,\chi}=\epsilon^{\prime\prime}_{l}$, so the relevant quotient in $G_2(s,\chi)$ is
\[
\frac{\Gamma_{\C}\left(s+\frac{k_p-1}{2}+\nu_p\right)}{\Gamma_{\R}(s+\epsilon^{\prime\prime}_{l}+\nu_{l}^{\prime})}=P_l(s)\Gamma_{\R}(s+\epsilon^{\prime}_{l}+\nu_{l}^{\prime}).
\]
Thus, for either character $\chi$ we have
\[
\frac{\Gamma_{\C}\left(s+\frac{k_p-1}{2}+\nu_p\right)}{\Gamma_{\R}(s+\epsilon^{\prime}_{l,\chi}+\nu_{l}^{\prime})}
=P_l(s)\Gamma_{\R}(s+\epsilon^{\prime\prime}_{l,\chi}+\nu_{l}^{\prime}).
\]
for $\epsilon^{\prime\prime}_{l,\chi}=\epsilon^{\prime}_{l,\chi}
+1\pmod 2$ and some polynomial $P_l(s)$ with zeros in the half-plane $\re(s)<1/2$. Thus, if the factor
$\Gamma_{\R}(s+\epsilon^{\prime}_{l,\chi}+\nu_{l}^{\prime})$ in the denominator can be cancelled by a factor in the numerator for the trivial character, the corresponding twisted factors can also be cancelled, leaving a polynomial factor behind.
\item[Case 3.] If $t_1+1\le l\le t_2$, we may carry out our analysis in the same way as before. If there exists $p=p(l)$ with $r_1+1\le p\le r_1+r_2+1$ such that 
$\Gamma_{\C}\left(s+\frac{k_p-1}{2}+\nu_p\right)$ and $\Gamma_{\C}\left(s+\frac{k_l^{\prime}-1}{2}+\nu_p^{\prime}\right)$, one sees immediately that $\frac{k_p-1}{2}+\nu_p=\frac{k_l^{\prime}-1}{2}+\nu_p^{\prime}+m$ for some $m\in \Z$, then 
\[
\Gamma_{\C}\left(s+\frac{k_p-1}{2}+\nu_p\right)=P_l(s)^{\pm 1}\Gamma_{\C}\left(s+\frac{k_l^{\prime}-1}{2}+\nu_p^{\prime}\right)
\]
for some polynomial $P_l(s)$ of degree $m$ according to whether $m$ is non-negative or negative. If no such $p$ exists, we see that we are simply analysing quotients which are reciprocal to the ones in Case 1 and Case 2. We see that either the factors in the denominator cancel out completely, or we obtain the reciprocal of a polynomial function.

\end{enumerate}
The analysis above shows that if all but finitely many poles of $L(s,\rho_{\infty})$ are poles of $L(s,\pi_{\infty})$, then all but finitely many poles of $L(s,\rho_{\infty}\times\chi_{\infty})$ must be poles of $L(s,\pi_{\infty}\times\chi_{\infty})$. Thus
if $G_2(s)$ has finitely many zeros, then $G_2(s,\chi)$ must also have only finitely many zeros.

\end{proof}

\begin{lemma} \label{correctgamma} Assume that $G_2(s)$ has only finitely many zeros.
Then there exists an irreducible unitary representation $\tau_{\infty}$ of $\gltwor$ satisfying the following properties.
\begin{enumerate}
\item For every primitive Dirichlet character $\chi$ there are polynomials $p(s,\chi):=p(s,\chi_{\infty})$ 
and $q(s,\chi):=q(s,\chi_{\infty})$ which are coprime, such that all the zeros of $p(s,\chi)$ and $q(s,\chi)$ lie in the half-plane $\re(s)<1/2$ and
\begin{equation}\label{quotfneqn}
G_2(s,\chi)=R(s,\chi)L(s,\tau_{\infty}\times\chi_{\infty}),
\end{equation}
where $R(s,\chi)=p(s,\chi)/q(s,\chi)$. 
\item When $\tau_{\infty}$ is a discrete series representation or a limit of discrete series representation, $p(s,\chi)$ and $q(s,\chi)$ can be chosen so that the pole of $L(s-1,\tau_{\infty}\times\chi_{\infty})$ with largest real part is not a zero of $p(s,\chi)$ or $q(s,\chi)$ (hence, all the zeros of $p(s,\chi)$ and $q(s,\chi)$ lie in the half-plane $\re(s)<1/2$).
\end{enumerate}
\end{lemma}
\begin{proof}
Again, we emphasise that because of our assumptions on $\pi_{\infty}$ and $\tau_{\infty}$, $F_2$ satisfies (G3).
By Lemma \ref{gtwochifin} we know that
$G_2(s,\chi)$ has only finitely many zeros. We reprise the arguments of the previous lemma.
In each of the cases  that we considered, we obtained either a polynomial or the reciprocal of a polynomial for the relevant quotients. We may thus proceed inductively starting with $l=1$, and succesively cancel the gamma functions appearing in the denominator with those in the numerator in \eqref{twistedquotg}. Since the degree of the numerator is $2$ more than the degree of the denominator, two factors of degree $1$, or a single factor of degree $2$, will remain in the numerator. Again, by the duplication formula for the gamma function, we may view a factor of degree $2$ as the product of two degree $1$ factors.
Thus, after canceling all the factors in the denominator out and relabelling the indices if necessary, we see that
\begin{equation}\label{gltwogprelim}
G_2(s,\chi)=\begin{cases} R_0(s)\Gamma_{\R}\left(s+\epsilon_{1,\chi}+\nu_1\right)\Gamma_{\R}\left(s+\epsilon_{2,\chi}+\nu_2\right)\quad\text{or}\\
R_0(s)\Gamma_{\C}\left(s+\frac{k-1}{2}+\nu_0\right).\end{cases}
\end{equation}
where $R_0(s):=R_0(s,\chi)=p_0(s)/q_0(s)$ is a quotient of polynomials, $\epsilon_{i,\chi}\in \{0,1\}$, $-1/2<\re(\nu_i)<1/2$, $ i= 1,2$, $k\in \N$, and where $\chi=\chi_{\infty}$ is either the trivial or sign character. We further note that all the zeros and poles of $R_0(s)$ lie in the half-plane
$\re(s)<1/2$. Further, we can obviously assume that $p_0(s)$ and $q_0(s)$ are coprime.

We analyse the first case a little further.
Since $\tau$ and $\rho$ are both assumed unitary, we know that 
$\sum_{p=1}^{r_1+r_2}\re(\nu_p)=\sum_{l=1}^{t_1+t_2}\re(\nu_l^{\prime})=0$ by \eqref{archunit}. It follows that $\re(\nu_1)+\re(\nu_2)=0$.
If $\epsilon_{1,\chi}+\nu_1\ne\epsilon_{1,\chi}+\nu_2\pm 1$, we
see that $\re(\nu_1)\ne 0$. In this case, we see that $\nu_1=\nu+ib_1$ and $\nu_2=-\nu+ib_2$ for $\nu,b_1,b_2\in \R$ and we obtain
\[
G_2(s,\chi)=h_{\chi}\pi^{-s}R_0(s)\Gamma_{\R}\left(s+\epsilon_{1,\chi}+\nu+ib_1\right)\Gamma_{\R}\left(s+\epsilon_{2,\chi}-\nu+ib_2\right).
\]
If $\epsilon_{1,\chi}+\nu+ib_1=\epsilon_{1,\chi}-\nu\pm 1+ib_2$, then $b_1=b_2$, and  we may assume (after relabelling the two factors if necessary) that 
$2\nu=\epsilon_{2,\chi}-\epsilon_{1,\chi}-1$. Since $-1/2<\re(\nu)<1/2$, it follows that $\epsilon_{1,\chi}=0$, $\epsilon_{2,\chi}=1$ and $\re(\nu)=0$. If $\nu_1=ib_1$, $b_1\in \R$, the duplication formula now shows that 
\[
G_2(s,\chi)=R_0(s)\Gamma_{\C}\left(s+ib_1\right).
\]
Thus, the first case and the second case of \eqref{gltwogprelim} coincide.

A similar analysis of the second case of \eqref{gltwogprelim} shows that we must have $\re(\nu_0)=0$, so $\nu_0=ib_2$ for $b_2\in \R$. Thus, in all cases we may rewrite \eqref{gltwogprelim} as
\begin{equation}\label{gltwogrational}
G_2(s,\chi)=\begin{cases} R(s)\Gamma_{\R}\left(s+\epsilon_{1,\chi}+\nu+ib_1\right)\Gamma_{\R}\left(s+\epsilon_{2,\chi}-\nu+ib_2\right)\,\,\text{or}\\
R_0(s)\Gamma_{\C}\left(s+\frac{k-1}{2}+ib_3\right)\end{cases}
\end{equation}
for some $b_1,b_2,b_3\in \R$. 

Equation \eqref{ginftytwo} shows that $L(s,\tau_{\infty})$ necessarily has one of the forms in \eqref{gltwogrational} upto the rational function $R_0(s)$. Conversely, it is not hard to see that given a factor having one of the forms in \eqref{ginftytwo}, there is a unitary representation $\tau_{\infty}$ such that $L(s,\tau_{\infty})$ is exactly this factor.
In the first case $\tau_{\infty}$ will be a principal series representation, while in the second it will be a (limit of) discrete series representation (see page 97 of \cite{JaLa70}). Moreover, the zeros and poles of $R_0(s)$ lie in the region $\re(s)<-1/2$. Thus, we have established the first assertion of the lemma. When $\tau_{\infty}$ is a principal series representation, we take $R(s)=R_0(s)$.

To prove the second assertion, if $\alpha$ is a zero of $q_0(s)$ and $\alpha=1-\frac{k-1}{2}-ib_3$, the pole of $\Gamma_{\C}\left(s-1+\frac{k-1}{2}+ib_3\right)$ with the largest real part, then
\[
p_0(s)\frac{\Gamma_{\C}\left(s+\frac{k-1}{2}+ib_3\right)}{q_0(s)}=p_0(s)\frac{\Gamma_{\C}\left(s+\frac{k-3}{2}+ib_3\right)}{q_1(s)},
\]
where $q_1(s)=(s-\alpha)q_0(s)$. If $\alpha-1$ is a zero of $q_1(s)$, we can similarly absorb it into the gamma function and remove the factor $(s-\alpha+1)$ from $q_1(s)$. We may proceed inductively until all such factors are cancelled from $q_0(s)$ to leave behind the polynomial $\tilde{q}(s,\chi)$. By a similar argument we may cancel factors from $p_0(s)$ and absorb them into the gamma function until no zeros of $p_0(s)$ remain which are poles of $L(s-1,\tau_{\infty}\times\chi_{\infty})$. We call the resulting polynomial $\tilde{p}(s,\chi)$. We cancel any common factors between $\tilde{p}(s,\chi)$ and $\tilde{q}(s,\chi)$, and call the resulting polynomials $p(s,\chi)$ and $q(s,\chi)$ respectively. Choose $R(s,\chi)=p(s,\chi)/q(s,\chi)$. This completes the proof of the lemma.
\end{proof}

\begin{lemma}\label{gtwofinite} Assume that $\pi=\otimes_v^{\prime}\pi_v$ and $\rho=\otimes_v^{\prime}\rho_v$ are irreducible admissible representations of $\gln$ and $\glntwo$ respectively, with $\pi_{\infty}$ and $\rho_{\infty}$ unitary. Assume that $L(s,\pi)$ and $L(s,\rho)$ satisfy (P1), (P4) and functional equations of the form \eqref{autlfneqn} for the trivial character. Suppose that $F_2$ has at most a finite number of poles. Then $G_2$ has at most a finite number of zeros.
\end{lemma}
\begin{proof}
Since $L(s,\pi)$ and $L(s,\rho)$ satisfy (P1) and (P4), it follows that $F_2$ satisfies (P1) and (P4). Since we have assumed that 
$\pi_{\infty}$ and $\tau_{\infty}$ are unitary, they satisfy \eqref{gselbound}. Hence, $F_2$ satisfies (G3). Using the duplication formula for the gamma function, we see that 
\[
\Gamma_{\C}(s+a)=\Gamma_{\R}\left(s+a\right)
\Gamma_{\R}\left(s+a+1\right).
\]
Note that the poles and zeros of $G_2(s,\chi)$ lie on lines parallel to the $x$-axis and are unions of sets of the form 
\[
S_p=\{-\epsilon_{p}-\nu_p-2m\,\vert\, m\in \Z, m\ge 0\}\,\,\text{and}\,\,S_{l}=\{-\epsilon^{\prime}_{l}-\nu_{l}^{\prime}-2m\,\vert\, m\in \Z, m\ge 0\}
\]
respectively. 

Since $F_2$ satisfies (P4), it is non-vanishing in some right half-plane. It follows that $\Phi_2(s)=G_2(s)F_2(s)$ is non-vanishing in some right half-plane. The functional equation \eqref{quotfneqn} shows that $\Phi_2(s)=E_2(s)\widetilde{\Phi_2}(1-s)$, where $E_2(s)$ is the product of a constant and an exponential function. Hence, it is non-vanishing in some left half-plane. 
Thus, all but finitely many of the zeros of $G_2$ must be poles of $F_2$ in some left half-plane. Since $F_2$ has only a finite number of poles, it follows $G_2$ can have at most a finite number of zeros in this left half-plane. Further, $G_2$ has only finitely many zeros in any right half-plane. It follows that $G_2$ has only finitely many zeros.
\end{proof}

\begin{remark} The proof above shows that one does not really need to assume that $L(s,\pi)$ and $L(s,\rho)$ satisfy (P4). Non-vanishing (of the quotient $F_2$) in some right half-plane is enough.
\end{remark}

\begin{proposition} \label{ftwopolyfneqnthm}
Assume that the hypotheses of Lemma \ref{gtwofinite} are satisfied. Suppose also that for every primitive Dirichlet character $\chi$, $F_2(s,\chi)$ has a meromorphic continuation to all of $\C$ and satisfies the functional equation
\begin{equation}\label{secquotfneqn}
\Phi_2(s,\chi)=G_2(s,\chi)F_2(s,\chi)=E_2(s,\chi)\widetilde{\Phi_2}(1-s,\chi),
\end{equation}
with $E_2(s,\chi)$ as defined in \eqref{etwodefn}. Then 
\begin{equation}\label{ftwopolyfneqn}
 \Phi_2(s,\chi)=p(s,\chi)\tilde{q}(1-s,\chi)L(s,\tau_{\infty}\times\chi_{\infty})
 F_2(s,\chi)=E_2(s,\chi)\widetilde{\Phi_2}(1-s),
\end{equation}
with all the zeros of $p(s,\chi)$ and $q(s,\chi)$ lying the in half-plane $\re(s)<1/2$.
\end{proposition}
\begin{proof}
This is an immediate consequence of Lemma \ref{correctgamma} and the hypotheses of the theorem.
\end{proof} 

\begin{remark} Note that the proof of the lemma above requires us to know very little about the holomorphy of $L(s,\pi)$ and $L(s,\rho)$ individually in $\C$. Only the quotient is required to have a meromorphic continuation. In practice, we usually know much more. If $\pi$ and $\rho$ arise from Artin representations or certain other classes of Galois representations, we know that the individual $L$-functions have a meromorphic continuation to all of $\C$. When dealing with $L$-functions associated to automorphic representations, Theorems \ref{lautoing} and \ref{rsthm} give much stronger holomorphy results.
\end{remark}

\section{Eliminating the polynomial function from the functional equation of $F_2(s)$}\label{archgltwo}

In this section we will largely follow the notation and exposition in \S 5 of \cite{JaLa70} while making the changes relevant to our situation. The main result of this section (Theorem \ref{testvecforps}) is a purely local statement about the $L$-functions of representations of $\gltwor$.

Let $\tau_{\infty}$ be an irreducible unitary (admissible) representation of $\gltwor$. It gives rise to an admissible representation of $\{\ug,\varepsilon\}$,
where $\ug$ is the universal envelopping algebra of $\lieg_{\C}=\lieg\otimes_{\R}\C$, $\lieg=\liegltwo$ and
\[
\varepsilon=\begin{pmatrix} -1&0\\\,\,\,0&1\end{pmatrix}.
\]
By abuse of notation we will use $\tau_{\infty}$ to denote the representation of $\{\ug,\varepsilon\}$ as well.
Let $\chi_j(t)=\sgn^{\epsilon_j}\vert t\vert^{\nu_j}$, where $\sgn$ denotes the sign character, $\epsilon_j\in \{0,1\}$, and $j=1,2$, be a character of $\R^{\times}$. It is known that every irreducible admissible representation $\tau_{\infty}$ is necessarily isomorphic to a representation parabolically induced from a pair $(\chi_1,\chi_2)$ of characters on $\R^{*}$ viewed as a character of the Borel subgroup
(these are the principal series representations denoted by $\pi(\chi_1,\chi_2)$ in \cite{JaLa70}) or a certain irreducible quotient of this induced representation (Theorem 5.11 of \cite{JaLa70}). The latter are the discrete series or limits of discrete series representations. In what follows, we will assume that the representation $\tau_{\infty}$ is generic. This results in no loss of generality since every infinite-dimensional unitary representation of $\gltwor$ is generic.

If $\varphi\in {\mathcal S}(\R^2)$, the Schwartz space of $\R^2$, we let
\[
\ft{}{\varphi}(\xi)=\int_{\R^2}\varphi(\eta)e^{2\pi i\langle \eta,\xi\rangle}d\eta
\]
denote the Fourier transform of $\varphi$, where $\xi,\eta\in \R^2$. We also denote by 
\[
\ft{1}{\varphi}(x,y)=\int_{\R}\varphi(u,y)e^{2\pi iux} du,\,\,\text{and}\,\,\ft{2}{\varphi}(x,y)=\int_{\R}\varphi(x,v)e^{2\pi ivy} dv,
\]
the Fourier transforms of $\varphi$ in just the first and just the second variables respectively. Note that if $\varphi(x,y)=\varphi_1(x)\varphi(y)$, then
\[
\ft{}{\varphi}(x,y)=\ft{1}{\varphi_1}(x)\ft{2}{\varphi_2}(y).
\]
We can associate to any $\varphi\in {\mathcal S}(\R^2)$ the function $W_{\varphi}(g)$ on $\gltwor$ as follows:
\[
W_{\varphi}\left[\begin{pmatrix} a&0\\
                  0&1
                 \end{pmatrix}\right]
=\chi_1(a)|a|^{\frac{1}{2}}\int_{\R^{*}}\chi_1\chi_2^{-1}(t)\ft{2}{\varphi}(at,t^{-1})d^{\times}t.
\]
This allows us to obtain an embedding of $\tau_{\infty}$ into a space of functions $W:\gltwor\to \C$
such that 
\[
W\left[\begin{pmatrix} 1&x\\
                  0&1\end{pmatrix}g\right]=\psi_{\infty}(x)W[\,g]
\]
on which $\gltwor$ acts by right translation, where $\psi_{\infty}(x)=e^{2\pi ix}$. We let $\whit(\tau_{\infty}):=\whit(\tau_{\infty},\psi_{\infty})$ denote the 
Whittaker model of $\tau_{\infty}$, that is, the image of $\tau_{\infty}$ under this embedding.
For $W\in \whit(\tau_{\infty})$, we define
\[
 \Psi(s,W)=\int_{\R^{\times}}W\left[\begin{pmatrix} a&0\\
                                 0&1
                                \end{pmatrix}\right]|a|^{s-\frac{1}{2}}
                                d^{\times}a.
\]
Let $\varphi_0(x,y)=e^{-\pi(x^2+y^2)}$. We will be 
interested in the Whittaker functions corresponding to the choice of
Schwartz function $\varphi_{n}(x,y)=(x-iy)^{n}\varphi_0(x,y)$.
when $n\ge 0$. We will use the notation $W_{n}:=W_{\varphi_{n}}$ and these will be $K$-finite vectors in $\whit(\tau_{\infty})$, where $K$ is maximal compact subgroup of $\gltwor$. Recall that $W_{n}$ appears in the Whittaker model of $\tau_{\infty}$ if and only if $n\equiv \epsilon(\tau_{\infty})=\epsilon:=\epsilon_1+\epsilon_2 \pmod 2$. 

Let $\varphi^{(1)}_k(x)=x^ke^{-\pi x^2}$ and $\varphi^{(2)}_l(y)=y^le^{-\pi y^2}$, $k,l\in \Z_{\ge 0}$. We see that we can write.
\[
\varphi_{n}(x,y)=
\sum_{m=0}^n\binom{n}{m}\varphi^{(1)}_m(x)(-i)^{n-m}\varphi^{(2)}_{n-m}(y)
\]

We will require the following elementary proposition.
\begin{proposition}\label{pcoeff} We have
\[
\ft{1}{\varphi^{(1)}_k}=i^k\varphi^{(1)}_k\quad\text{and}\quad \ft{2}{\varphi^{(2)}_l}=i^l\varphi^{(2)}_l.
\]
Hence,
\[
\ft{}{\varphi_n}=i^n\varphi_n\quad\text{and}\quad\ft{2}{{\varphi}_{n}}(x,y)=(x+y)^n\varphi_0(x,y)\quad\text{and}\quad
\]
\end{proposition}
\begin{theorem}\label{testvecforps}
Let $\tau_{\infty}$ be an irreducible admissible representation of $\gltwor$ and let $\whit(\tau_{\infty})$ denote its Whittaker model. Given any polynomial $P(s)$, there exists a Whittaker function $W(P)\in \whit(\tau_{\infty})$ such that 
\[
\frac{\Psi(s,W(P))}{L(s,\tau_{\infty})}=P(s).
\]
\end{theorem}
\begin{proof} Before embarking on the proof, we note that Theorem 5.15 of
\cite{JaLa70} tells us that for any $W\in \whit(\tau_{\infty})$ the quotient
$\Psi(s,W)/L(s,\tau_{\infty})$ is necessarily an entire function. The point is that we can realise any polynomial $P(s)$ as such a quotient for a suitable choice of $W$.

We will prove the theorem assuming that $\tau_{\infty}$ is a principal series representation associated to the pair of characters $(\chi_1,\chi_2)$. The proofs in the other cases are entirely analogous and easier.
Now the function $\ft{2}{\varphi}_n(at,t^{-1})$ is a linear combination with positive (integral) coefficients of terms of the form
\[
a^mt^{2m-n}e^{-\pi(a^2t^2+t^{-2})},
\]
where $0\le m\le n$. 
By our remarks above, we see that $\Psi(s,W_{n})$ is a sum with positive integral coefficients of integrals of the form
\begin{equation}\label{zetasexplicit}
\int_{\R^{\times}}\chi_1(a)|a|^{s}a^m
\int_{\R^{\times}}\chi_1\chi_2^{-1}(t)t^{2m-n}e^{-\pi(a^2t^2+t^{-2})}d^{\times}td^{\times}a,
\end{equation}
where $0\le m\le n$. 

Since $\epsilon=\epsilon_1+\epsilon_2 \pmod 2$ and $n$ have the same parity, we see that the inner integral
is unchanged under $t\mapsto -t$. 
If $\epsilon_1=0$, the outer integral is unchanged under $a\mapsto -a$ if $m$ is even, and zero otherwise, while if $\epsilon_1=1$, the outer integral is unchanged under $a\mapsto -a$ if $m$ is odd, and zero otherwise. Thus in both cases, the integral \eqref{zetasexplicit} is non zero only if $\epsilon_1$ and $m$ have the same parity. 
It follows that $\Psi(s,W_{n})$ is a linear combination with positive coefficients of terms of the form
\begin{equation}\label{nonzeroint}
4\int_0^{\infty}a^{\nu_1+m+s}\int_{0}^{\infty}t^{\nu_1-\nu_2+2m-n}
e^{-\pi(a^2t^2+t^{-2})}d^{\times}td^{\times}a,
\end{equation}
where $\epsilon_1$ and $m$ have the same parity. After the change of variable $t^2/a\mapsto t$, and using the formula 
\[
\int_0^{\infty}\int_0^{\infty}e^{-a(t+t^{-1})/2}t^{\nu}d^{\times}ta^sd^{\times} a=2^{s-2}\Gamma\left(\frac{s+\nu}{2}\right)\Gamma\left(\frac{s-\nu}{2}\right)
\]
for the Mellin transform of Macdonald $K$-Bessel function, the term
\eqref{nonzeroint} becomes
\begin{flalign}\label{gammaforwn}
&4\int_0^{\infty}a^{\frac{\nu_1+\nu_2+n}{2}+s}\int_{0}^{\infty}t^{\frac{\nu_1-\nu_2+2m-n}{2}}
e^{-2\pi a(t+t^{-1})/2}d^{\times}td^{\times}a\nonumber\\
&=(2\pi)^{-\frac{\nu_1+\nu_2+n}{2}}\pi^{-s}\Gamma\left(\frac{s+\nu_1+m}{2}\right)\Gamma\left(\frac{s+\nu_2+n-m}{2}\right).
\end{flalign}
We can recover the $L$-function
$L(s,\pi_{\infty})$ from \eqref{gammafactor} as follows.
\begin{enumerate}
 \item If $\epsilon=0$ and $\epsilon_1=\epsilon_2=0$, we take $n=0$. 
 \item If $\epsilon=0$ and $\epsilon_1=\epsilon_2=1$, we take $n=2$. In this case, only the term corresponding to $m=1$ will yield a non-zero term.
 \item If $\epsilon=1$, $\epsilon_1=1$ and $\epsilon_2=0$, we take
$n=1$. In this case, only the term corresponding to $m=1$ will yield a non-zero term.
\end{enumerate}
The last possibility -- $\epsilon=1$, $\epsilon_1=0$ and $\epsilon_2=1$, is already covered by the third case. Thus, with the choices above, we have in all cases
\[
\Psi(s,W_n)=L(s,\tau_{\infty}).
\]
Going back to equation \eqref{gammaforwn}, we see that 
the functional equation for the gamma function shows that 
the expression in \eqref{gammaforwn} has the form $Q(s)L(s,\tau_{\infty})$ for some polynomial $Q(s)$ with leading term
$s^{[n/2]}(2\pi)^{n/2}$ in cases (1) and (3) above, and 
$s^{[n-2/2]}(2\pi)^{n/2}$ in case (2). Recall that $\Psi(s,W_n)$ is a linear combination with positive integral coefficients of such terms and is thus once again
has a leading term of the same exponent.
Since $n$ was arbitrary subject only to the condition $n\equiv \epsilon \pmod 2$, it follows that $\Psi(s,W_n)$ form a basis for the space functions which are products of polynomials and $L(s,\tau_{\infty})$, 
and our result follows immediately by choosing $W(P)$ to be a suitable linear combination of the functions $W_n$. 
\end{proof}
As before, let $\chi$ be a Dirichlet character, so $\chi_{\infty}$ is either the trivial or sign character.
Recall that the Whittaker model  $\whit(\tau_{\infty}\otimes\chi_{\infty},\psi_{\infty})$ consists of functions
$W\otimes\chi_{\infty}$, where $W\in \whit(\tau_{\infty},\psi_{\infty})$.  
We will need the local functional equation which we now recall.
If $w=\begin{pmatrix}0&1\\1&0\end{pmatrix}$ is the Weyl element, define
\[
\reallywidecheck{\Psi}(s,W)=\int_{\R^{\times}}W\left[\begin{pmatrix} a&0\\
                                 0&1
                                \end{pmatrix}w\right]\omega_{\infty}(a)|a|^{s-\frac{1}{2}}d^{\times} a.
\] 
                            
The integrals $\Psi(s,W\otimes\chi_{\infty})$ satisfy the local functional equation (Theorem 5.13 of \cite{JaLa70})
\begin{equation}\label{firstlocfneqn}
\frac{\reallywidecheck{\Psi}(1-s,W\otimes\chi_{\infty}^{-1})}{L(1-s,\tilde{\tau}_{\infty}\times\chi_{\infty}^{-1})}=\varepsilon(s,\tau_{\infty}\times\chi_{\infty},\psi_{\infty})\frac{\Psi(s,W\otimes\chi_{\infty})}{L(s,\tau_{\infty}\times\chi_{\infty})},
\end{equation}
where $\tilde{\tau}_{\infty}$ is the representation contragredient to 
$\tau_{\infty}$, and $\varepsilon(s,\tau_{\infty}\times\chi_{\infty}
,\psi_{\infty})$ has the form $r(\tau_{\infty}\times\chi_{\infty},
\psi_{\infty})K(\tau_{\infty}\times\chi_{\infty},\psi_{\infty})^s$ for 
$r(\tau_{\infty}\times\chi_{\infty},\psi_{\infty})\in \C$ and
$K(\tau_{\infty}\times\chi_{\infty},\psi_{\infty})>0$ . For the choice of
character $\psi_{\infty}(x)=e^{2\pi ix}$ (made at the start of this section, as well as before equation \eqref{epsiloninftytwo}), we have $K(\tau_{\infty}\times\chi_{\infty},\psi_{\infty})=1$ and $r(\tau_{\infty}\times\chi_{\infty},\psi_{\infty})=r(\tau_{\infty},\psi_{\infty})$ a constant (independent of $\chi_{\infty}$) given by \eqref{epsiloninftytwo}. Thus, we have
\begin{equation}\label{locfneqn}
\frac{\reallywidecheck{\Psi}(1-s,W\otimes\chi_{\infty}^{-1})}{L(1-s,\tilde{\tau}_{\infty}\times\chi_{\infty}^{-1})}=r(\tau_{\infty},\chi_{\infty},\psi_{\infty})\frac{\Psi(s,W\otimes\chi_{\infty})}{L(s,\tau_{\infty}\times\chi_{\infty})},
\end{equation}

Assume now that $\tau_{\infty}$ and $\chi_{\infty}$ are both unitary. Then
we can write $\nu_1=\nu+ib_1$ and $\nu_2=-\nu+ib_2$ for $\nu, b_1,b_2\in \R$.
It follows that 
\[
L(1-s,\tilde{\tau}_{\infty}\times\chi_{\infty}^{-1})=\widetilde{L}(1-s,\tau_{\infty}\times\chi_{\infty}).
\]
From \eqref{gammaforwn} we see that for $W_n$, 
\[
\reallywidecheck{\Psi}(1-s,W_{n}\otimes \chi_{\infty}^{-1})=\widetilde{\Psi}(1-s,W_{n}\otimes\chi_{\infty}),
\]
and the same holds for the Whittaker function $W(P)$ for any polynomial $P$, since it is simply a linear combintation of the functions $W_n$. 
Using Theorem \ref{testvecforps} for $\tau_{\infty}\otimes\chi_{\infty}$, we can find a vector $W^{\prime}=W(p(s,\chi)\tilde{q}(1-s,\chi))$ which is a linear combination of vectors $W_n\otimes\chi_{\infty}$. Using this vector in the local functional equation \eqref{locfneqn} yields
\begin{equation}\label{modlocfneqn}
\frac{L(1-s,\tilde{\tau}_{\infty}\times\chi_{\infty}^{-1})}{\widetilde{\Psi}(1-s,W^{\prime}\otimes\chi_{\infty})}=r(\tau_{\infty},\chi_{\infty},\psi_{\infty})^{-1}\frac{L(s,\tau_{\infty}\times\chi_{\infty})}{\Psi(s,W^{\prime}\otimes\chi_{\infty})}.
\end{equation}
Muliplying \eqref{ftwopolyfneqn} by the equation above, we obtain the following theorem.
\begin{theorem}\label{correctfneqnthm} Let $\pi$ and $\rho$ satisfy the hypotheses of Lemma \ref{gtwofinite}. There is an irreducible unitary representation $\tau_{\infty}$ of $\gltwor$ such that for any
 primitive Dirichlet character $\chi$,
\begin{equation}\label{correctfneqn}
 \Phi_2(s,\chi)=L(s,\tau_{\infty}\times\chi_{\infty})
 F_2(s,\chi)=r(\tau_{\infty},\chi_{\infty},\psi_{\infty})^{-1}
 E_2(s,\chi)\widetilde{\Phi_2}(1-s,\chi),
\end{equation}
where $E_2(s,\chi)$ is defined in \eqref{etwodefn}.
\end{theorem}

\section{The epsilon factor associated to $F_2(s,\chi)$}\label{secepsilon}

Theorem \ref{correctfneqnthm} tells us that the factor $G_2(s,\chi)$ that  appears in the functional equation of $F_2(s,\chi)$ is of $\gltwog$-type, that is, it has the same form as the archimedean $L$-function of a cuspidal automorphic representation of $\gltwo$. We will choose a Dirichlet character $\chi_0$ which is
sufficiently highly ramified at a certain finite set of primes and we set
$F_2(s,\chi_0)=F_0(s)$. In this section we will show that if $\chi$ is a
Dirichlet character ramified only at primes not in the chosen finite set,
the factor 
$E_0(s,\chi)=E_2(s,\chi_0\times\chi)$ that appears in the functional equation of
$F_0(s,\chi)$ is also of $\gltwog$-type. We will need one more ingredient from the local representation theory to accomplish this.

Assume that $\sigma=\otimes_v^{\prime}$ is an irreducible admissible representation of $\gln$ with a
central character $\omega_{\sigma}$. Let $\omega_{\sigma,v}$ be its component at the place $v$.
As before, $S_{\sigma}$ will denote the set of primes at which the representation $\sigma$ is ramified. 
The notation $L(s,\sigma,\chi)$ will denote the twist
of the Dirichlet series $L(s,\sigma)$ by $\chi$ as defined in \eqref{selchartwist} while
$L(s,\sigma\times\chi)$ denotes the twist in the sense of automorphic $L$-functions.
Suppose $S_{\sigma}$ is the set of primes where
$\sigma$ is ramified. Let $\chi$ be a Dirichlet character which is
unramified at all the primes of $S_{\sigma}$. It is not hard to 
see that $L(s,\sigma,\chi)=L(s,\sigma\times\chi)$. If we allow $\chi$ to be ramified at the
places of $S_{\sigma}$ the picture is more complicated, and is described by a theorem of Jacquet 
and Shalika on the ``stability of 
gamma factors'' (see the Proposition in (2.2) of \cite{JaSh85}) which we recall below. Part of the content of their theorem  is 
that the two kinds of twisting we have defined are the same provided that $\chi$ is sufficiently highly ramified at the places 
where $\sigma$ is ramified.
\begin{theorem}\label{stability}[Jacquet-Shalika] There exists an integer $M_{\sigma}$,
such that for all primitive Dirichlet characters $\chi_0$ with ramification greater than 
$p^{M_{\sigma}}$ at the places of $p\in S_{\sigma}$, the following hold:
\begin{enumerate}
\item $L(s,\sigma_p\times \chi_{0,p})=L(s,\sigma_p,\chi_{0,p})=1$, and 
\item $\varepsilon(s,\sigma_p\times\chi_{0,p},\psi_p)=\varepsilon(s,\chi_{0,p},\psi_p)^{n-1}\varepsilon(s,\omega_{\sigma,p}\chi_{0,p},\psi_p)$.
\end{enumerate}
\end{theorem}
\begin{remark} Analogues of the theorem above have been proved for the tensor product $L$-functions (a consequence of the Local Langlands Conjecture in \cite{HaTa01}, for instance), and for the exterior and symmetric square $L$-functions in \cite{CoShTs2017}.
\end{remark}

We recall that for a Dirichlet character $\chi_1$ of conductor $D_1$, 
\[
\varepsilon(s,\chi_1)=(-i)^{\epsilon_{\chi_1}}\tau(\chi_1) D_1^{-s}
\]
where $\tau(\chi_1)$ is the Gauss sum associated to $\chi_1$ and the additive character $\psi$. This follows easily from 
the definitions of the local $\varepsilon$-factors given in \eqref{epschip} and \eqref{epschiinfty}.
If $\chi_2$ is a Dirichlet character of conductor $D_2$ with $(D_1,D_2)=1$, one sees that 
\[
\varepsilon(s,\chi_1\chi_2)=\chi_1(D_2)\chi_2(D_1)\varepsilon(s,\chi_1)\varepsilon(s,\chi_2).
\]
We apply the theorem above to the representations $\pi$ and $\rho$ satisfying the hypotheses of Lemma \ref{ftwopolyfneqn} and having automorphic central characters. Let $S_{\pi}$ and $S_{\rho}$ denote the respective set of primes where $\pi$ and $\rho$ are ramified and $M_{\pi}$ and $M_{\rho}$ be chosen as in Theorem \ref{stability}. Finally, set 
$M=\max\{M_{\pi},M_{\rho},K\}+1$. We choose a primitive Dirichlet character $\chi_0$ such that $\chi_{0,p}$ has conductor $p^M$ for all $p\in S_{\pi}\bigcup S_{\rho}$, so the conclusions of Theorem \ref{stability} hold for each of the representations $\pi$ and $\rho$. Let the conductor of $\chi_{0}$ be $N_0=\prod_{i=1}^kp^M$. Then, the conductors of $\pi\times\chi_0$ and $\rho\times\chi_0$ are $N_0^n$ and $N_0^{n-2}$ respectively. Let $\chi$ be a (primitive) Dirichlet character of conductor $D$ which is unramified at all the places in $S=S_{\pi}\bigcup S_{\rho})$. Using Theorem \ref{stability}, we have
\[
\varepsilon(s,\pi_{p}\times\chi_{0,p}\chi_p)=\varepsilon(s,\chi_{0,p}\chi_p)^{n-1}\varepsilon(s,\omega_{\pi,p}\chi_{0,p}\chi_p)
\]
and 
\[
\varepsilon(s,\rho_{p}\times\chi_{0,p}\chi_p)=\varepsilon(s,\chi_{0,p}\chi_p)^{n-3}\varepsilon(s,\omega_{\rho,p}\chi_{0,p}\chi_p)
\]
for all $p\in S_{\pi}\cup S$. The corresponding global $\varepsilon$-factors are
\[
\varepsilon(s,\pi\times\chi_0\chi)=\omega_{\pi}(D)\chi_0(D)^n
\varepsilon(s,\omega_{\pi}\chi_0)\varepsilon(s,\chi_0)^{n-1}\chi(N_0)^n\varepsilon(s,\chi)^n
\]
and 
\[
\varepsilon(s,\rho\times\chi_0\chi)=\omega_{\rho}(D)\chi_0(D)^{n-2}\varepsilon(s,\omega_{\rho}\chi_0)\varepsilon(s,\chi_0)^{n-2-k}\chi(N_0)^{n-2}\varepsilon(s,\chi)^{n-2}.
\]
Let $N=N_0^2$. Taking the quotients of the two $\epsilon$-factors yields
\begin{equation}\label{ezerofactor}
E_0(s,\chi):=E_2(s,\chi_0\chi)=c_{\chi_0}\omega_{\pi}\omega_{\rho}^{-1}\chi_0^2(D)N^{1/2-s}\chi(N)\varepsilon(s,\chi)^2,
\end{equation}
where 
\[
c_{\chi_0}=\frac{r(\omega_{\pi}\chi_0,\psi)}{r(\omega_{\rho}\chi_0,\psi)}.
\]
does not depend on $\chi$ or $D$. We recall that the archimedean $\varepsilon$-factor is indepenedent of the character $\chi_{\infty}$ by which it is twisted. Hence, the factor$r(\tau_{\infty},\chi_{\infty},\psi_{\infty})^{-1}$ that appears in \eqref{correctfneqn} is indepependent of $\chi$. Combining our calculations above with \eqref{correctfneqn}, we see that
the twists $F_0(s,\chi):=F(s,\chi_0\chi)$ by primitive Dirichlet characters $\chi$ satisfy the functional equations
\begin{equation}\label{finalfneqn}
\Phi_0(s,\chi)=L(s,\tau_{\infty})F_0(s,\chi)=c_{\chi_0}^{\prime}\omega_{\pi}\omega_{\rho}^{-1}(D)N^{1/2-s}\chi(N)\varepsilon(s,\chi)^2
\tilde{\Phi}_0(1-s,\chi),
\end{equation}
where $c_{\chi_0}^{\prime}$ does not depend on $\chi$.

\section{A $\gltwog$ converse theorem and the proof of Theorem \ref{zerothm}}\label{conversesec}

The final ingredient we will require is a version of Weil's converse theorem valid for all Dirichlet series satisfying a
$\gltwog$-type functional equation.

We recall that if $F(s)=\sum_{n=1}^{\infty}a_nn^{-s}$ and $\chi$ is a Dirichlet character, we have defined $F(s,\chi)=\sum_{n=1}^{\infty}\chi(n)a_nn^{-s}$. If $S$ is a finite set of non-archimedean places of $\Q$, we define 
\[
F_S(s)=\sum_{(n,p)=1\,\text{if $p\in S$}}\frac{a_n}{n^{s}}.
\]
This is obviously consistent with the earlier notation for partial $L$-functions.

\begin{theorem}\label{weilplus} Let $\tau_{\infty}$ be an irreducible unitary representation of $\gltwor$. We denote by $k$ the weight of $\tau_{\infty}$ (resp. parity) if it is a (limit of) discrete series representation (resp. principal series representation). Let $D\in \N$ be fixed and let $N\in \N$ with $(N,D)=1$ be arbitrary. Let $\psi$ be a Dirichlet character modulo $N$ and assume that $\psi(-1)=(-1)^k$. 
We will assume that $F(s)$ is a Dirichlet series satisfying (P1), (P2), and the following condition.
\begin{enumerate} 
\item[(W3)] There exists a Dirichlet series $F_1(s)=\sum_{n=1}^{\infty} b_nn^{-s}$ satisfying (P1), such that for every primitive Dirichlet character $\chi \pmod D$ the functions
\[
\Phi(s,\chi)=L(s,\tau_{\infty}\times\chi_{\infty})F(s,\chi)\,\,\text{and}\,\, L(s,\tilde{\tau}_{\infty}\times\chi_{\infty}^{-1})F_1(s,\chi)
\]
are entire and satisfy functional equations of the form
\begin{equation}\label{weilfneqn}
\Phi(s,\chi)=\psi(D)N^{1/2-s}\chi(N)\varepsilon(s,\chi)^2\tilde{L}(1-s,\sigma_{\infty}\times\chi_{\infty})F_1(1-s,\bar{\chi}).
\end{equation}
\end{enumerate}
Then there exist distinct automorphic representations $\tau_i$, $1\le i\le r$, with $L(s,\tau_{i,\infty})=L(s,\tau_{\infty})$ for all
$1\le i\le r$, such that 
\begin{equation}\label{flincomb}
F_S(s)=\sum_{i=1}^{r}c_iL_S(s,\tau_i)
\end{equation}
for some constants $c_i$, $1\le i\le r$.
\end{theorem}

\begin{remark} A slightly stronger version of the theorem above was formulated and proved by Weil in \cite{Weil67} when $\sigma$ is a discrete series or a limit of discrete series representation. In this case Weil's theorem asserts that 
$f(z)=\sum_{n=1}^{\infty} \chi_0(n)a_ne^{2\pi inz}$ is a modular form of level $N$ and nebentypus $\psi$ of a suitable weight $k$. Since the modular eigenforms give a basis of the space of all modular forms of level $N$, nebentypus $\psi$ and weight $k$, and since each eigenform corresponds to an automorphic representation of $\gltwo$, we see that the conclusion of our theorem follows.
\end{remark}
\begin{remark}
When $\tau_{\infty}$ is a principal series representation, Weil's proof essentially goes through with minimal changes. Weil's proof hinges on producing an elliptic element $M\in \sltwor$ of infinite order such that $(g\vert_{\gamma}-\psi(D)g)\vert_M=g\vert_{\gamma}-\psi(D)g$, where $g(z)=\sum_{n=1}^{\infty}b_ne^{2\pi inz}$, the candidate modular form of which $F_1(s)$ will be the Mellin tranform,  for all $\gamma$ in a set of elements generating $\Gamma_0(N)$. This can be deduced from the functional equations for sufficiently many character twists of $F$.
On the other hand he shows that any holomorphic function $f$, such that $f\vert_M=f$ must be identically zero. Although, $f$ is holomorphic in Weil's theorem, the same argument works for real analytic $f$ as well, which is what the candidate automorphic form will be in general. Indeed, what one uses is only the fact that the torus generated by $M$ has a limit point and the principle of analytic continuation for power series.
\end{remark}
\begin{remark} Note that in one sense the functional equation \eqref{weilfneqn} is actually more general than the one satisfied by the elements of $\gsel$ because it allows the Dirichlet series $F_1(s)$ to be arbitrary.
\end{remark}
\begin{remark} As is well known, Theorem \ref{weilplus} requires only that the conditions on the twists $\Phi(s,\chi)$ hold for a suitably chosen finite set of primitive characters $\chi$ (see \cite{IIPS75}). Invoking this stronger form of the theorem does not simplify our proof.
\end{remark}
We are now in a position to prove Theorem \ref{mostgeneral} of this paper, from which Theorem \ref{zerothm} will follow almost immediately.
\begin{theorem} \label{mostgeneral} Let $\pi=\otimes_{v}^{\prime}\pi_v$ and $\rho=\otimes_{v}^{\prime}\rho_v$ be (global) irreducible unitary representations with automorphic central characters of $\gln$ and $\glntwo$ respectively, and with $L$-functions $L(s,\pi\times\chi)$ and $L(s,\rho\times\chi)$ having meromorphic continuations and satisfying functional equations of the form $\eqref{autlfneqn}$ for every primitive Dirichlet character $\chi$. Suppose that $F_2(s):=L(s,\pi)/L(s,\rho)\in \gselp$, that $G_2$ has at most finitely many zeros, and that it satisfies (G5). Then there exists $\tau\in \auttwoall$, and a finite set of finite places $S$ of $\Q$ such that 
\begin{equation}\label{ftwoaut}
F_S(s)=L_S(s,\tau)
\end{equation}
\end{theorem}
\begin{proof} 
Since $F_2\in \gselp$ and satisfies (G5) by assumption, we see that the hypotheses of Theorem \ref{bookermodthm} are satisfied. It follows that $F_2(s,\chi)$ is holomorphic in the strip $0<\re(s)\le 1$, for every non-principal character $\chi$. Further, $\Phi_2(s,\chi)$ is entire if $\chi\ne 1$ is primitive by Corollary \ref{phischientire}.

Let $\chi_0$ be the (primitive) character chosen in Section \ref{stability} and recall that $F_0(s)=F_2(s,\chi_0)$. By our remarks above, $F_0(s)$ is holomorphic in $0<\re(s)\le 1$ and $\Phi(s,\chi_0)$ is entire. Since $G_2$ has only finitely many zeros, it follows from Lemma \ref{gtwochifin} that $G_0(s)=G_2(s,\chi_0)$ has only finitely many zeros. It follows that $F_0(s)$ has only finitely many poles, since these poles (which lie outside the strip $0<\re(s)\le 1$) necessarily coincide with the zeros of $G_0(s)$, since $\Phi_0(s)=G_0(s)F_0(s)$ is entire.

Thus, we can now assume that $F_0(s)=L(s,\pi\times\chi_0)/L(s,\rho\times\chi_0)$ extends meromorphically to all of $\C$ with at most finitely many poles. Replacing $F_2$ by $F_0$ in the arguments of the first paragraph of the proof, we see that 
$F_0\in \gselp$. Now, let $\chi$ be a primitive Dirichlet character with conductor coprime to the conductor of $\chi_0$.
Since $L(s,\pi\times\chi_0\chi),L(s,\rho\times\chi_0\chi)\in \gselp$ by Theorem \ref{lautoing}, we know that $F_0(s,\chi)=F_2(s,\chi_0\chi)$ has a meromorphic continuation to all of $\C$ and satisfies \eqref{secquotfneqn}. Thus $F_0$ satisfies all the hypotheses of Proposition  \ref{ftwopolyfneqnthm}, and hence satisfies a functional equation of the form \eqref{ftwopolyfneqn}. Since we have assumed that $\pi$ and $\tau$ are unitary, so are $\pi_{\infty}$ and $\tau_{\infty}$, and the hypotheses of Theorem \ref{correctfneqnthm} hold, whence we know that $F_0(s,\chi)$ satisfies an equation of the form 
\eqref{correctfneqn}. Since $\chi_0$ has been chosen as in Section \ref{stability}, we see that $F_0(s,\chi)$ satisfies a functional equation of the form \eqref{weilfneqn} with $\psi=\omega_{\pi}\omega_{\rho}^{-1}\chi_0^2$, and where $F_1=c_{\chi_0}^{\prime}\tilde{F}_0$. The automorphy hypotheses on the central characters guarantee that $\psi$ is a Dirichlet character upto a unitary twist. Once again, an appeal to Corollary \ref{phischientire} shows that $\Phi_0(s,\chi)$ is entire.

The function $F_0(s)$ thus satisfies all the hypotheses of Theorem \ref{weilplus}. It
follows that 
\[
F_{0,S}(s)=L_S(s,\pi\times\chi_0)/L_S(s,\rho\times\chi_0)=\sum_{i=1}^rc_iL_S(s,\tau_i^{\prime}),
\]
for suitable unitary automorphic representations $\tau_i^{\prime}\in \auttwoall$, $1\le i\le r$ with $L(s,\tau_{i,\infty}^{\prime})=L(s,\tau_{\infty}^{\prime})$. Since $F_{2,S}(s)=F_{0,S}(s,\chi_0^{-1})$, we have
\begin{equation}\label{fnlincomb}
F_{2,S}(s)L_S(s,\rho)=L_S(s,\pi)=L_S(s,\rho)\left[\sum_{i=1}^rc_iL_S(s,\tau_i)\right],
\end{equation}
for $\tau_i=\tau_i^{\prime}\times\chi_0^{-1}$.

Since $\pi$ is admissible and unramified outside $S$, there exists a simultaneous eigenfunction $f$ on $\gln$ of all the Hecke operators $T_p$, $p\notin S$, such that $L(s,\pi)=L(s,f)$. Similarly, there exist eigenfunctions $f_i$ on $\gln$ such that $L(s,\rho)L(s,\tau_i)=L(s,f_i)$. At all places $p$ outside of $S$, we see that we have a sum of eigenfunctions $f_{i}$ which is once again an eigenfunction $f$. It follows that the eigenvalues corresponding to $f_{i}$ must all be the same for all $1\le i\le r$.
Hence, the representations $\tau_{i,p}$, $1\le i\le r$ which are characterised by their Satake parameters must all be the same representation $\tau_p$ at all $p\notin S$. By comparing the constant coefficient of the Dirichlet series on both sides of \eqref{fnlincomb}, we see that $\sum_{i=1}^rc_i=1$, and the theorem follows. 
\end{proof}

We now deduce Theorem\ref{zerothm} from Theorem \ref{mostgeneral}
\begin{proof}[Proof of Theorem \ref{zerothm}]
As we have already observed before, we can assume that $G_2(s)$ has at most a finite number of zeros.
If $\rho\in \temprep_{n-2}$, we know that 
$L(s,\rho)=\prod_{j=1}^kL(s,\rho_j)$ with $\rho_j\in \aut$. 
If $F_2$ has a finite number of poles in $\C$, Theorem \ref{lautoing} together with Theorem \ref{nonvanishthm} shows that $F_2\in \gselp$. By Theorem \ref{sigmaponeptwo} we know that $L(s,\pi)$ and $L(s,\rho)$ satisfy functional equations of the form \eqref{autlfneqn} and by Theorem \ref{nonvanishthm} we see that it satisfies (G5), so all the hypotheses of the Theorem \ref{mostgeneral} are satisfied. It follows
that $L_S(s,\pi)=L_S(s,\rho)L_S(s,\tau)$ for some $\tau\in \auttwoall$.
Twisitng both sides by $\tilde{\tau}$, we see that the right hand side will have a pole at $s=1$ while the right hand side will be holomorphic by \cite{JaSh811}. This proves the theorem.
\end{proof}
\begin{remark} We could have also deduced Theorem \ref{zerothm} from \eqref{fnlincomb} without using \eqref{ftwoaut} as follows. Once again, using \cite{JaSh811}, we see that each of the $L$-functions $L(s,\pi)$ and $L(s,\rho)L(s,\tau_i)$, $1\le i\le r$ are distinct. Then the main result of \cite{KMP06} says that these $L$-functions must be linearly independent, contradicting \eqref{fnlincomb}.
\end{remark}

\section{Applications to the Primitivity of cuspidal $L$-functions}\label{secprim}
 
In this section we will deduce Theorem \ref{primthm} from Corollary \ref{generaln}. We will need previously known classification results for the classes $\asel$ and $\gselp$ to do this. In \cite{Ragh20} and \cite{BaRa20} these are often formulated for series in the larger class $\asel$. We sometimes state only the versions relevant to the class $\gselp$ below.

Recall that $\asel_d$ (resp. $\gselp_d$) is the class of series in $\asel$ (resp. $\gselp$) of degree $d$. The four important classfication theorems for elements of $\gselp$ of small degree recorded below
are Theorems 4.3, 4.7 and 5.1 of \cite{Ragh20} and Theorem 1.1 of \cite{BaRa20} respectively.
\begin{theorem}\label{gselpzero} We have $\gselp_0=\{1\}$.
\end{theorem}
\begin{theorem}\label{gselpzeroone}  If $0<d<1$, then $\aseld=\emptyset$.
\end{theorem}
\begin{theorem}\label{gselpone} The set $\gselp_1$ consists of the functions of the form
$L(s+iA,\chi)$, where $A\in \R$ and $\chi$ is a primitive Dirichlet character.
\end{theorem}
\begin{theorem}\label{gselponetwo}
If $1<d<2$, then $\aseld=\emptyset$.
\end{theorem}
An easy consequence of Theorems \ref{gselpzero} and \ref{gselpzeroone} is Corollary 6.2 of
\cite{Ragh20}:
\begin{theorem}\label{factorisation} Every element of $\gselp$ can be factored in primitive elements.
\end{theorem}
\begin{proof}[Proof of Theorem \ref{primthm}]
The theorems above allow us to reduce the proof of Theorem \ref{primthm} to
Corollary \ref{generaln}. Indeed, suppose $\pi\in \autthree$. Then $L(s,\pi)\in \gselp_3$. By Theorem
\ref{factorisation}, we know that $L(s,\pi)$ must factor into primitive elements. By Theorems \ref{gselpzero} and  \ref{gselpzeroone}, the degree of the smallest non-trivial factor must be at least $1$. 
Thus, there can be at most three non-trivial factors. There can be no factor of degree between $1$ and $2$ by Theorem \ref{gselponetwo}, so we see that the only possible factorisations are
\[
L(s,\pi)=f_1(s)f_2(s)f_3(s),
\]
where $f_j(s)\in \gselp_1$ for $1\le j\le 3$, or
\[
 L(s,\pi)=f_1(s)f_2(s),
\]
with $f_1(s)\in \gselp_1$ and $f_2(s)\in \gselp_2$. 
In either case, Theorem \ref{gselpone} asserts that $f_1(s)=L(s+iA_1,\chi_1)$ for $A_1\in \R$, and a primitive Dirichlet character $\chi_1$. We have 
\begin{equation}\label{threefactors}
L(s,\pi)=L(s+iA_1,\chi_1)F_2(s),
\end{equation}
where $F_2(s)\in \gselp_2$ with $F_2=f_2f_3$ in the first case above, or $F_2=f_2$ in the second case. If $\rho=\chi_1\norm^{-iA}\in \autone$, $L(s,\pi)=L(s,\rho)F_2(s)$, contradicting Corollary
\ref{generaln}.
\end{proof}

\begin{remark} We require only that $\gselp_d$ be empty for $1<d\le 3/2$ for our argument to go through. 
\end{remark}
\begin{remark} If one restricts oneself to $\sel$, the analogues of Theorems \ref{gselpzero}, \ref{gselpzeroone} and \ref{gselpone} for $\esel$ can be found in \cite{KaPe99}. In \cite{KaPe03}, the authors prove that $\esel_d=\emptyset$ for $1<d<5/3$, which would be adequate if one wanted to prove Theorem \ref{primthm} for the class $\sel$. The analogue of Theorem \ref{gselponetwo} for 
$\esel$ was proved in \cite{KaPe11}.
\end{remark}
\begin{remark} Motivated by the problem of establishing the primitivity of $L(s,\pi)$ for $\pi\in \autthree$, I was able to prove a version of Theorem \ref{primthm} several years ago for the case $n=3$ and the class $\sel$. But the result did not appear to be general enough, since most cuspidal $L$-functions of $\glthree$ are not known to belong to $\sel$. This was my iniitial motivation for the introduction of the classes $\asel$ and $\gsel$.
\end{remark}
\begin{remark} One could work with a notion of ``almost primitive'' elements, defined as the product of a primtive element of degree greater than zero and an element of degree $0$, and formulate our results of the class $\asel$. The class $\gselp$ however, allows the cleanest formulation of primitivity results. 
\end{remark}

\section{Comparing the zero sets of $L$-functions}\label{zerosets}

Theorem \ref{zerothm} asserts that there are infinitely many zeros
(counted with multiplicty) of $L(s,\rho)$ 
which are not zeros (counted with multiplicty) of $L(s,\pi)$. More precisely, let $m_{\theta}$
denote the multiplicity of a zero $\theta$ of $L(s,\sigma)$, where 
$\sigma\in \temprep$. We let 
\[
S_{\sigma}=\{(\theta,m_{\theta})\,\vert\, L(\theta,\sigma)=0\}.
\]
Then the theorem asserts that $\vert S_{\rho}\setminus S_{\pi}\vert=\infty$.

When $G_2$ has finitely many zeros, the theorem tells us that there are infinitely many non-trivial zeros of $L(s,\rho)$ which are not zeros of $L(s,\pi)$ in the critical strip. Let 
$S_{c,\sigma}=\{s\in S_{\sigma}\,\vert\,0<\re(s)<1\}$. We record this reformulation separately as
\begin{corollary}\label{zerosincs} Let $\pi\in \autn$ and let $\rho\in \temprep_{n-2}$. If $G_2(s)=L(s,\pi_{\infty})/L(s,\rho_{\infty})$ has at most a finite number of zeros, then 
$|(S_{c,\rho}\setminus S_{c,\pi})|=\infty$.
\end{corollary}
\begin{proof} We know that the poles of $F_2$ outside the critical strip are necessarily zeros of $G_2$, since $\Phi_2(s)$ is holomorphic when $\re(s)\le0$. Since there are at most finitely many of these, we see by Theorem \ref{zerothm} that there are infinitely many poles of $F_2$ in the critical strip.
\end{proof}
In \cite{Ragh99} and \cite{Ragh10}, we treated the special case of the corollary above when $G_2(s)=L(s,\tau_{\infty})$ for some unitary representation $\tau_{\infty}$ of $\gltwor$, that is, when $R(s)=1$, and when $F_2(s)$ had conductor $1$ (or, more precisely, when $\pi$ and $\rho$ had the same conductor).

Special cases of $\pi\in \autn$ and $\rho\in \temprep_{n-1},\temprep_n$
were treated in \cite{Ragh99}, and then considerably generalised in
\cite{Booker2015} and still further generalised in Theorem 7.1 of \cite{Ragh20}. This last theorem has no restrictions on the archimedean factors or even on the sizes of the Satake parameters, and is thus applicable to the tensor product, exterior square and symmetric square $L$-functions. We refer to the introduction of \cite{Ragh99} for the previous history of comparing zero sets of $L$-functions.

We give one class of examples to illustrate the generality of Corollary \ref{zerosincs} and how it improves on the existing results.

\begin{example} Let $\tau$ be a unitary cuspidal automorphic representation of $\gltwo$. Let 
$\pi={\sym}^n\tau$, $n\ge 0$, and $\rho=\sym^{n-2}\tau$, $n\ge 2$, the automorphic symmetric power lifts of 
$\tau$ (if they exist). These examples are particularly interesting for the following reason. When $\tau_p$ is unramified
 let $\alpha_p$ and $\alpha_p^{-1}$ be its Satake parameters. In this case the local Euler factor of $F_2(s)$ at $p$ has the form
\[
F_{2,p}(s)=(1-\alpha_p^{n+1}p^{-s})^{-1}(1-\alpha_p^{-n-1}p^{-s})^{-1}.
\]
Thus, we may view the local Euler factor as the local $L$-function of an admissible representation of $\gltwoqp$. 

When $\tau_{\infty}$ is a principal series representation (this case corresponds to representations associated to Maass cusp forms), the symmetric power lifts ${\sym}^n\tau\in \aut_{n+1}$ are known to exist for $n\le 4$ (\cite{GeJa76}, \cite{KiSh02a}, \cite{Kim03}). In this case, one checks easily that $G_2(s)$ has the form
\[
\Gamma\left(\frac{s+\varepsilon +ni\nu}{2}\right)\Gamma\left(\frac{s+\varepsilon -ni\nu}{2}\right)
\]
which clearly has no zeros. Kim and Shahidi give precise criteria for the cuspidality of $\pi$ in \cite{KiSh02a}. In this case Corollary \ref{zerosincs} can be stated as
\begin{corollary}  Let $\tau\in \auttwo$ such that $\tau_{\infty}$ is a
 principal series representation and assume that $\pi={\sym}^n\tau\in \autn$ for $2\le n\le 4$, and that $\rho={\sym}^{n-2}\tau$. Then, $|(S_{c,\rho}\setminus S_{c,\pi})|=\infty$.
\end{corollary}
When $F_2$ has conductor 1, this was treated in \cite{Ragh10} and when $n=2$, this is the work of \cite{NeOl2020}.

When $\tau_{\infty}$ is a (limit of) discrete series representation and $\tau$ is not CM, the $n$-th symmetric power lifts are known to exist and be cuspidal for all $n\in \N$ by the work of Newton and Thorne \cite{NeTh21}. As observed in Remark \ref{ramanujanrmk}, these symmetric power $L$-functions satisfy the Generalised Ramanujan Conjecture at all places.  
In this case, we can check that $G_2(s)$ has the form $P(s)L(s,\tau_{\infty})$ for some polynomial $P(s)\not\equiv 1$ if $n\ge 4$, so $G_2(s)$ has at most finitely many zeros. 
We record this special case of Corollary \ref{zerosincs} as
\begin{corollary} \label{sympowers} Let $\tau\in \auttwo$ such that $\tau_{\infty}$ is a (limit of) discrete series representation and $\tau$ is not CM, and let $\pi={\sym}^n\tau$ and $\rho={\sym}^{n-2}\tau$. Then for $n\ge 3$, $\vert S_{c,\rho}\setminus S_{c,\pi}\vert=\infty$.
\end{corollary}
Note that for the corollary above we really do need to allow for a polynomial factor to appear in the expression for $G_2(s)$.
To summarise, for every $n\ge 3$, the Euler product of $F_2(s)$ in the case above has the same form as the Euler product of a a global admissible representation of $\gltwo$, so $F_2(s)$ is the (partial) $L$-function of an admissible representation of $\gltwo$. It also satisfies (G1), (G3) and (G4) but has infinitely many poles (in the critical strip) and thus does not satisfy (G2). 
\end{example}

\begin{remark} The case $n=2$ in the corollary above requires a somewhat different argument which we will give in Section \ref{gtwoinfty}.
\end{remark}

\begin{remark} In \cite{Ragh99} we had stated the result above for the pair $\pi={{\rm Sym}}^4(\tau_{\Delta})$ and $\rho={{\rm Sym}}^2(\tau_{\Delta})$ where $\tau_{\Delta}$ is the cuspidal automorphic representation attached to $\Delta$, but the proof was incomplete. Corollary \ref{zerosincs} allows us to correct this mistake.
\end{remark}

\section{Relaxing the condition of automorphy}\label{relaxaut}

In this section we will give examples of pairs of $L$-functions $L(s,\pi)$ and $L(s,\rho)$ which have not been shown to be automorphic but for which we can nonetheless prove that $F_2(s)=L(s,\pi)/L(s,\rho)$ has infinitely many poles (in the critical strip).
Theorem \ref{mostgeneral} is applicable in a very wide variety of situations and will allows us to conclude that outside of a finite number of primes, $F_2$ coincides with the $L$-function of an automorphic representation.
It is possiblle to give a large number of pairs of $L$-functions for which we can extend the results of Theorem \ref{zerothm}. We content ourselves with two sets of examples, the first involving Artin $L$-functions, and the second involving the tensor product $L$-functions.

\begin{corollary}\label{artin} Let $\pi$ be an irreducible three dimensional Artin representation and let $\rho$ be a (unitary twist of) a Dirichlet character, and suppose that $G_2$ has only finitely many zeros. Then 
$\vert S_{c,\rho}\setminus S_{c,\pi}\vert=\infty$.
\end{corollary}
\begin{proof} Artin representations give rise to global admissible representations of $\gln$. The irreducibility of $\pi$ implies that $F_2$ satisfies (G5), and the other hypotheses of Theorem \ref{mostgeneral} follow easily from the well known properties of Artin $L$-functions. It follows from Theorem \ref{mostgeneral} that for some finite set of places $S$, we have
\begin{equation}\label{artinaut}
L_{S}(s,\pi)=L_S(\rho)L_S(s,\tau).
\end{equation}
Twisting both sides of the equation above by $\rho^{-1}$, we see that the left hand side is holomorphic on the line $\re(s)=1$, while the right hand side has a pole at $s=1$, yielding a contradiction.
\end{proof}
The corollary above generalises the work of Hochfilzer and Oliver in \cite{HoOl2022} where $\rho$ is taken to be the trivial character and $\pi$ is in a somewhat restricted subset of representations. One reason we are able to handle arbitrary characters $\rho$ is our use of the stability results of Theorem \ref{stability}.

Corollary \ref{artin} indicates a possible strategy for showing the irreducibility of certain Artin, or more generally, Galois representations, since the reducibility of a representation implies that its $L$-function factorises as a product Galois $L$-functions. However, I was unable to come up with any concrete non-trivial examples where this strategy might be exploited.

Corollary \ref{artin} permits the following generalisation.
\begin{corollary} Let $\pi$ be an irreducible Artin representation of dimension $n$ and let $\rho\in \autntwoall$. If $G_2$ has at most finitely many zeros, $\vert S_{c,\rho}\setminus S_{c,\pi}\vert=\infty$.
\end{corollary}
\begin{proof} 
 The proof is almost identical to the previous case. As before we get
 \[
L_S(s,\pi)=L_S(s,\rho)L(s,\tau)
 \]
by using Theorem \ref{mostgeneral}.  Thus, $\pi=\rho\boxplus\tau$
outside of a finite number of places $S$. We twist both sides above by $\tilde{\pi}$ to obtain
\[
L_S(s,\pi\times\tilde{\pi})=L_S(s,\rho\times\tilde{\rho})
L_S(s,\rho\times\tilde{\tau})^2L_S(s,\tau\times\tilde{\tau}).
\]
The irreducibility of $\pi$ shows that $L_S(s,\pi\times\tilde{\pi})$ has a simple pole at $s=1$, while the right hand side has at least a double pole at $s=1$
\end{proof}

As a variant on the theme above, we can assume that $\rho$ is also an Artin representation of dimension $n-2$ but that $G_2(s)=(2\pi)^{-s}\Gamma(s)$. Then, $\tau$ will arise from a weight $1$ form, and under some additional hypotheses, will correspond to an Artin representation of dimension $2$. But this will contradict the irreducibility of $\pi$. Other variations involving the $L$-functions of (compatible families of) Galois representations can also be formulated.

We record a second situation in which Theorem \ref{zerothm} does not suffice, but where Theorem \ref{mostgeneral} works.
\begin{corollary} \label{tensprodzeros} Let $\pi_i\in {\mathcal A}_{n_i}^{\circ}$, $1\le i\le 3$, and suppose that $n_1n_2=n_3+2$ and $n_2n_3>6$ . Assume that $\pi_2$ is tempered at all places, and that there is at least one prime $p$ where $\pi_1$, $\pi_2$ and $\pi_3$ are all unramified, and such that at least one of the local parameters $\alpha_{\pi_3,p}$ (as in \eqref{nonarchlfn}) is not a parameter of $\pi_{1,p}$ or $\pi_{2,p}$. If $G_2(s)=L(s,\pi_{1,\infty}\times\pi_{2,\infty})/L(s,\pi_{3,\infty})$ has at most finitely many zeros, then $\vert S_{c,\pi_3}\setminus S_{c,\pi_1\times\pi_2}\vert=\infty$.
\end{corollary}
\begin{proof}
Note that if $n_1n_2\le 6$, the corollary already follows from Theorem \ref{zerothm}, since $L(s,\pi_1\times\pi_2)$ is known to be automorphic in these cases (\cite{Ramak00}, \cite{KiSh02a}). As we have recorded in Theorem \ref{rsthm}, $L(s,\pi_1\times\pi_2)$ satisfies all the hypotheses of Theorem \ref{mostgeneral}. Thus, if the quotient $F_2(s)$ has only finitely many poles in $\C$, there is finite set $S$ of primes such that
\begin{equation}\label{nonautlincomb}
L_S(s,\pi_1\times\pi_2)= L_S(s,\pi_3)L_S(s,\tau)
\end{equation}
with $\tau\in \auttwoall$. One checks easily in the proof of Theorem \ref{mostgeneral} that we choose $S$ so that $p\notin S$.
Since one of the local 
parameters $\alpha_{\pi_3,p}$ (as in \eqref{nonarchlfn}) is not a parameter of $\pi_{1,p}$ or $\pi_{2,p}$, we know that 
$L_S(s,\pi_3)L_S(s,\tau)\ne L_S(s,\pi_1\times\pi_2)$, giving a contradiction.
\end{proof}

In view of Remark \ref{exttwontrem}, we can formulate similar corollaries involving the exterior and symmetric square $L$-functions.
\section{An example when $G_2$ has infinitely many zeros}\label{gtwoinfty}

When $\pi,\rho\in \aut$ and $G_2$ has infinitely many zeros, we know that $F_2$ has infinitely many poles in $\C$, but our arguments thus far have not been able to show that infinitely many of these lie in the critical strip $0<\re(s)<1$. In this section, we will be able to obtain this more subtle conclusion for one class of pairs.

Let $\tau\in \auttwo$ such that $\tau$ is not CM and $\tau_{\infty}$ is a (limit of) discrete series representation of weight $k\in \N$. In this
case, we have
\begin{equation}\label{lsymtwoinfty}
L(s,{\sym}^2\tau_{\infty})=\Gamma_{\R}\left(\frac{s+k+1}{2}\right)\Gamma_{\R}\left(\frac{s+1}{2}\right)\Gamma_{\R}\left(\frac{s+k-1}{2}\right),
\end{equation}
while $L(s,{\sym}^0\tau_{\infty})=L(s,1_{\infty})=\Gamma(s/2)$, where $1$ denotes the trivial representation of $\glone$. If $k$ is odd,
\[
G_2(s)=\frac{L(s,{\sym}^2\tau_{\infty})}{L(s,{\sym}^0\tau_{\infty})}=[(s+k-2)(s+k-3)\cdots s]\Gamma_{\R}\left(\frac{s+k+1}{2}\right)\Gamma_{\R}\left(\frac{s+1}{2}\right)
\]
has only finitely many zeros, so this is a special case of Theorem \ref{zerothm}. When $k$ is even, $G_2$ has infinitely many zeros, and does not fit into our usual template. Nonetheless, we can proceed as follows.

As before, we denote the quotient $L(s,{\sym}^2\tau)/\zeta(s)$ by $F_2(s)$. Recall that ${\sym}^2\tau$ is cuspidal by the work of Gelbart and Jacquet in \cite{GeJa76}.
Suppose $F_2$ has finitely many poles. Then, by the arguments we have seen before, $F_2\in \gsel$ and it satisfies (G5). By Theorem \ref{bookermodthm}, $F_2(s,\chi)$ has no poles in the strip $0<\re(s)\le 1$ for every non-principal character $\chi$.

Let $\chi_1\ne 1$ be an even primitive Dirichlet character and let $F_3(s)=L(s,\chi_1)F_2(s)$. We let $G_3(s)=L(s,\chi_{1,\infty})G_2(s)=L(s,{\sym}^2\tau_{\infty})$.
Since $L(s,\chi)$ is entire, $F_3(s,\chi_1\times\chi)$ has no poles in the strip $0<\re(s)\le 1$ for any primitive Dirichlet character 
$\chi\ne 1$. Outside of the critical strip, the poles of $F_3(s,\chi_1\times\chi)$ are necessarily among the zeros of
$G_2(s,\chi_1\times\chi)$. It follows that $G_3(s,\chi_1\times\chi)F_3(s,\chi_1\times\chi)$ is entire.

Recall that $F_{2,p}(s)=(1-\alpha_p^{2}p^{-s})^{-1}(1-\alpha_p^{-2}p^{-s})^{-1}$ at every place where $\tau$ is unramified.
After twisting by a primitive Dirichlet character $\chi_0$ that is sufficiently highly ramified at the places $R$ where $\tau$ and $\chi_1$ are ramified, and unramified outside $R$, we can assume that $(\chi_{1,p}\times\chi_{0,p})\boxplus(\tau_p\times\chi_{0,p})$ is a ramified principal series representation, and that  $F_{3,p}(s,\chi_0)=1$ for $p\in R$. For any prime $p$, we set
\[
\pi_{3,p}=(\chi_{1,p}\times\chi_{0,p})\boxplus(\tau_p\times\chi_{0,p}).
\]
The factor $G_3(s,\chi_{0,\infty})$ determines an irreducible admissible representation $\pi_{3,\infty}$ of $\glthreer$. 
Let $\pi_3=\otimes_v\pi_v^{\prime}$. If $\chi$ is a Dirichlet character unramified at all the places of $R$, and $T$ is the 
set of places where it is ramified, $\pi_3\times\chi$ is unramified outside of $S=R\bigcup T$. It follows that $\pi_3\times\chi$ is an irreducible admissible representation of $\glthree$. Clearly $\Lambda(s,\pi_3\times \chi)=G_3(s,\chi)F_3(s,\chi)$ is entire of order $1$ and satisfies a functional equation of the form \eqref{autlfneqn}. Thus, all the hypotheses of the converse theorem of Jacquet, Piatetski-Shapiro and Shalika for $\glthree$ in  \cite{JPSS792} are satisfied, so $\pi_3$ is quasi automorphic. This leads to the equation
\[
L_S(s,{\sym}^2\tau)L_S(s,\chi_0)=\zeta_S(s)L_S(s,\pi_3^{\prime})
\]
for some $\pi_3^{\prime}\in \autthreeall$. The right hand side above obviously has a pole at $s=1$, while the left hand side is holomorphic at $s=1$, giving a contradiction. 
We have thus proved
\begin{theorem}\label{symsquarebyzeta} Let $\tau\in \auttwo$. Assume that $\tau$ is not CM and that $\tau_{\infty}$ is a (limit of) discrete series representation. Then $L(s,{\sym}^2\tau)/\zeta(s)$ has infinitely many poles in the strip $0<\re(s)<1$.
\end{theorem}

\bibliographystyle{alpha}
\bibliography{../../../../Bibtex/master2023}

\end{document}